\documentclass[QTF8,fontset=windows,final, leqno]{siamltex}
\usepackage{booktabs} 
\usepackage[noadjust]{cite}
\usepackage{ulem}
\usepackage{ctex}
\usepackage{amsmath}
\allowdisplaybreaks[4]
\usepackage{cancel}
\usepackage{multirow}
\usepackage{graphicx}
\usepackage{amssymb}
\usepackage{color}
\usepackage{mathtools}
\usepackage{tikz}
\usepackage{hyperref}
\usepackage{float}
\usepackage{caption}
\usepackage[toc,page]{appendix}
\usepackage{mathrsfs}
\usepackage{mathtools}
\DeclarePairedDelimiter{\norm}{\lVert}{\rVert}
\usepackage{latexsym, bm}
\usepackage{bm}
\numberwithin{equation}{section}
\newtheorem{remark}{Remark}[section]

\newcommand{\vertiii}[1]{{\left\vert\kern-0.23ex\left\vert\kern-0.23ex\left\vert #1
		\right\vert\kern-0.23ex\right\vert\kern-0.23ex\right\vert}}
\renewcommand\abstractname{Abstract}

\renewcommand\figurename{Figure}
\renewcommand\tablename{Table}

\normalsize

\renewcommand\refname{References}

\def\no{{\nonumber}}

\newcommand{\mNQ}{\mathbf{N}_{h}}
\newcommand{\mNu}{\mathcal{N}_{h}}
\newcommand{\mNQz}{\mathbf{N}}
\newcommand{\mNuz}{\mathcal{N}}
\newcommand{\muQ}{\mathbf{H}_h}
\newcommand{\muu}{\mu_h}
\newcommand{\muQz}{\mathbf{H}}
\newcommand{\muuz}{\mu}
\newcommand{\muQa}{\mathbf{H}_{1,h}}
\newcommand{\muua}{\mu_{1,h}}
\newcommand{\muQb}{\mathbf{H}_{2,h}}

\newcommand{\muQaz}{\mathbf{H}_{1}}
\newcommand{\muuaz}{\mu_{1}}
\newcommand{\muQbz}{\mathbf{H}_{2}}

\newcommand{\mLQ}{\mathcal{L}_{h}}
\newcommand{\mLu}{\mathcal{D}_{h}}

\newcommand{\id}[2]{\left\langle #1,#2\right\rangle_h}
\newcommand{\ia}[1]{ \norm{#1 }_{h}^2}
\newcommand{\ilQ}[1]{\norm{#1 }_{\mathcal{L}_h}^2}
\newcommand{\ilu}[1]{\norm{#1 }_{\mathcal{D}_h}^2}
\newcommand{\ib}[1]{\norm{#1 }_{H^1_{h}}^2}

\newcommand{\iaQ}[1]{\norm{#1 }^2_{\mathcal{Q}_{\mathcal{L}}}}
\newcommand{\ibQ}[1]{\norm{#1 }^2_{\mathcal{Q}_{\mathcal{L}}^1}}
\newcommand{\icQ}[1]{\norm{#1 }^2_{\mathcal{Q}_{\mathcal{L}}^2}}
\newcommand{\iau}[1]{\norm{#1 }^2_{\mathcal{Q}_{\mathcal{D}}}}
\newcommand{\ibu}[1]{\norm{#1 }^2_{\mathcal{Q}_{\mathcal{D}}^1}}

\newcommand{\bQ}{\mathbf{Q}}
\newcommand{\bu}{\mathbf{u}}
 \newcommand{\mQ}{\mathcal{Q}(\tau \mLQ)}
\newcommand{\dtau}{\delta_\tau}
\newcommand{\dt}{\tau} 

\newcommand{\ep}{\epsilon}

\newcommand{\bM}{\mathbf{M}}

 \newcommand{\ea}{\delta_t \boldsymbol{E}_{\mathbf{Q}}^{n+1}}
\newcommand{\eb}{\delta_t \mu_{h}^{n+1}}
\newcommand{\ec}{\delta_t e_{s}^{n+1}}

 \newcommand{\eQ}{ \boldsymbol{E}_{\mathbf{Q}}^{n+1}}

 \newcommand{\da}{\delta_t \mathbf{Q}(t_{n+1})}
 \newcommand{\db}{\delta_t u(t_{n+1})}
\def\ba{\mathbf{a}}

\def\bu{\mathbf{u}}

\def\bQ{\mathbf{Q}_h}
\def\bu{u_h}

\newcommand{\mq}{\mathcal{Q}}

\def\ma{\mathcal{A}}

\def\ts{\tilde{s}}

 \def\ba{\begin{equation}\begin{aligned}}
 \def\ed{\end{aligned}\end{equation}}

\CTEXoptions[today=old]
\begin{document}


\renewcommand{\thefootnote}{\fnsymbol{footnote}}
\title{GENERALIZED SAV-EXPONENTIAL INTEGRATOR SCHEMES FOR A Modified Landau-de Gennes Theory for Smectic Liquid Crystals \footnote{Last update: \today}}

\author{Wenshuai Hu
	\and
	Guanghua Ji\thanks{Laboratory of Mathematics and Complex Systems, Ministry of Education and School of Mathematical Sciences,
		Beijing Normal University, Beijing 100875, China.\newline
				\texttt{202431130052@mail.bnu.edu.cn}, \texttt{ghji@bnu.edu.cn, Corresponding author}}
\and Xiao Li\thanks{Laboratory of Mathematics and Complex Systems, Ministry of Education and School of Mathematical Sciences,
		Beijing Normal University, Beijing 100875, China.}
}

\maketitle
\begin{abstract}
The macroscopic state of Smectic-A (SmA) phases is governed by a modified Landau-de Gennes (mLdG) model, which couples a tensor-order parameter $\mathbf{Q}$ for the orientational order with a real scalar $u$ representing the positional density deviation.
In this paper, we propose and analyze a novel, highly efficient, and unconditionally energy-stable numerical scheme for this  coupled system. Our approach systematically integrates the exponential time differencing (ETD) method to exactly resolve the high-order stiff linear spatial operators, and the generalized scalar auxiliary variable  approach to efficiently treat the severe nonlinear phase couplings. In particular, the ETD time discretization is reformulated into an equivalent implicit backward Euler-type structure, a pivotal step that eliminates the restrictive CFL mesh-ratio conditions of the original method and enables a rigorous fully discrete error analysis.
Theoretically, we rigorously establish the unconditional energy stability with respect to a modified discrete energy and the uniform boundedness of the numerical tensor $\mathbf{Q}$, along with optimal error estimates that demonstrate the high-order accuracy in both time and space.
 Comprehensive numerical experiments are presented to demonstrate the accuracy, efficiency, and structural preservation of the algorithm, as well as its capability in capturing complex topological defect dynamics.
\end{abstract}

\begin{keywords}
 Smectic-A liquid crystals; $Q$-tensor model; Scalar Auxiliary Variable (SAV); Exponential Time Differencing (ETD); Energy stability; Maximum bound principle; Error estimates.
\end{keywords}

\pagestyle{myheadings}
\thispagestyle{plain}
\markboth{W. Hu,
	G. Ji \and X. Li}{EOP-GSAV-EI chemes for Smectic Liquid Crystals}
\section{Introduction}~\\
Liquid crystals are a unique state of soft matter that transition through various mesophases-primarily the nematic, chiral, and smectic phases-based on their degree of molecular organization \cite{deGennes1974,hicks2024modelling,wang2021modelling}. The smectic phase, in particular, further subdivides into distinct types such as Smectic-A (SmA) and Smectic-C (SmC) \cite{han2015microscopic}.
 Among these, SmA liquid crystals feature a dual-order structure, combining one-dimensional positional layering with an orientational director aligned parallel to the layer normal \cite{xia2021structural,xia2023variational}. Concurrently describing the distinct partial orderings of fluid-like orientation and solid-like layering, along with their intricate interplay, presents formidable mathematical and computational challenges. Within the modified Landau--de Gennes theory \cite{shi2025modified}, the orientational order is described by a macroscopic, symmetric, traceless $Q$-tensor that naturally accommodates topological defects and phase transitions, coupled with a real scalar density variation field $u$ characterizing the smectic layer structure. We concentrate on the $Q-u$ theory for the SmA liquid crystals flows in this paper.

 The modified Landau-de Gennes (mLdG)  energy \cite{shi2025modified,xia2021structural,xia2023variational} for the  SmA phase is  given by
\begin{equation}
    E(\mathbf{Q}, u) = \int_{\Omega} \left( f_{n}(\mathbf{Q}, \nabla \mathbf{Q}) + f_{s}(u) + f_{int}(\mathbf{Q}, u) \right) \, \mathrm{d}\mathbf{x},
\end{equation}
where $\Omega$ is a smooth, bounded domain in $\mathbb{R}^d (d=2,3)$, $u$ represents the scalar density variation associated with the smectic layering, defined  in the space
  $   \mathcal{S}_u^{(d)} \overset{\mathrm{def}}{=} \left\{ u : \Omega \to \mathbb{R} \mid u \in H^2(\Omega) \right\},$
 and  $Q$ describes the orientational order of the nematic phase,  defined in the space
\begin{equation}
    \mathcal{S}_{\mathbf{Q}}^{(d)} \overset{\mathrm{def}}{=} \left\{ M \in \mathbb{R}^{d \times d} \;\middle|\; \operatorname{tr}(M) = 0, M^{ij} = M^{ji} \in \mathbb{R} \ \forall i,j = 1, \dots, d \right\},
\end{equation}
where $\operatorname{tr}(M) := \sum_{i=1}^d M^{ii}$.
The nematic energy density $f_n$ is based on the LdG theory and consists of an elastic term and a bulk potential:
\begin{equation}
    f_{n}(\mathbf{Q}, \nabla \mathbf{Q}) = \frac{K}{2} |\nabla \mathbf{Q}|^2 + f_{bn}(\mathbf{Q}),
\end{equation}
where $K > 0$ is the elastic constant. The bulk energy density $f_{bn}(\mathbf{Q})$ governs the isotropic-nematic phase transition and is given by a polynomial expansion of the $Q$-tensor invariants:
\begin{equation}
    f_{bn}(\mathbf{Q}) =
    \begin{cases}
        \frac{A}{2}\mathrm{tr}(\mathbf{Q}^2) + \frac{C}{4}(\mathrm{tr}(\mathbf{Q}^2))^2, & \text{if } d=2,\\
        \frac{A}{2}\mathrm{tr}(\mathbf{Q}^2) - \frac{B}{3}\mathrm{tr}(\mathbf{Q}^3) + \frac{C}{4}(\mathrm{tr}(\mathbf{Q}^2))^2, & \text{if } d=3, \\
    \end{cases}
\end{equation}
where $A = \alpha_1(T - T_1^*)$ is the rescaled temperature in which $\alpha_1 > 0$ and $T_1^*$ is a characteristic liquid crystal temperature, and $B, C > 0$ are material-dependent bulk constants.

The smectic bulk energy density $f_{s}(u)$ describes the formation of the layered structure and is given by a polynomial expansion of the scalar order parameter $u$:
\begin{equation}
    f_{s}(u) = \frac{a}{2}u^2 + \frac{b}{3}u^3 + \frac{c}{4}u^4,
\end{equation}
where $a = \alpha_2(T - T_2^*)$ is a temperature-dependent parameter with $\alpha_2 > 0$, and $T_2^* < T_1^*$ is a critical material temperature related to N-S phase transition; $b, c > 0$ are material-dependent constants.

The coupling term $f_{int}(\mathbf{Q}, u)$ represents the interaction between the nematic director and the smectic layer normal, and is given by
\begin{align*}
f_{int}(\mathbf{Q}, u) &= \begin{cases}
B_0 |D^2 u|^2, & A \geqslant \frac{B^2}{27C}, d = 3 \text{ or } A \geqslant 0, d = 2, \\
B_0 \left| D^2 u + q^2 \left( \frac{\mathbf{Q}}{s_+} + \frac{\mathbf{I}_d}{d} \right) u \right|^2, & otherwise,
\end{cases}, \\
    s_{+} &= \begin{cases}
\frac{B+\sqrt{B^2-24AC}}{4C}, & A < \frac{B^2}{27C}, d = 3, \\
\sqrt{\frac{-2A}{C}}, & A < 0, d = 2,
\end{cases}
\end{align*}
where $B_0 > 0$ is the coupling strength, $q \approx 2\pi/l$ is the wave number associated with the smectic layer thickness $l$, and $\mathbf{I}_d$ is the $d \times d$ identity matrix.

The $L^2$-gradient flow dynamics in $\mathbb{R}^d$ ($d = 2, 3$) corresponding to the energy functional $E[\mathbf{Q}, u]$ leads to a coupled system of nonlinear PDEs for the tensor order parameter $\mathbf{Q}$ and the scalar density variation $u$. Since $\mathbf{Q}$ takes values in the constrained space $\mathcal{S}^{(d)}$, its variational derivative must be projected onto the space of symmetric and traceless matrices. The explicit tensor form of the coupled gradient flow equations is given by:
\begin{align}
   \frac{\partial Q^{ij}}{\partial t} &= - \left( \frac{\delta E}{\delta Q} \right)^{ij} + \lambda \delta^{ij} + \mu^{ij} - \mu^{ji}, \quad 1 \le i, j \le d, \label{1.3a}\\
    \frac{\partial u}{\partial t} &= -  \frac{\delta E}{\delta u},  \label{1.3b}
\end{align}
where  $\lambda$ and $\mu = (\mu^{ij})_{d\times d}$ are Lagrange multipliers corresponding to the tracelessness constraint and the matrix symmetry constraint, respectively. The resulting system  of \eqref{1.3a}-\eqref{1.3b} can be explicitly written as
\begin{equation}\label{eq1-9}
\begin{aligned}
    \frac{\partial \mathbf{Q}}{\partial t} &= K\Delta\mathbf{Q} - \left[ A\mathbf{Q} - B \left( \mathbf{Q}^2 - \frac{\mathrm{tr}(\mathbf{Q}^2)}{3}\mathbf{I} \right) + C \mathrm{tr}(\mathbf{Q}^2)\mathbf{Q} \right] \\
    &\quad - 2B_0 q^2 / s_+ \cdot \left( u \cdot D^2 u - \frac{\mathrm{tr}(u \cdot D^2 u)}{3}\mathbf{I} \right) - 2B_0 q^4 \cdot \frac{\mathbf{Q}}{s_+^2} u^2,\\
    \frac{\partial u}{\partial t} &= - 2B_0\Delta^2 u - au - bu^2 - cu^3  - 2B_0 D^2 u : q^2 M \\
    &\quad - 2B_0 q^2\nabla \cdot \left( \nabla \cdot \left( M u \right) \right)  - 2B_0 \cdot \left| q^2 M \right|^2 u,
\end{aligned}
\end{equation}
where $M=\frac{\mathbf{Q}}{s_+} + \frac{\mathbf{I}_d}{d}$, and $\Delta^2 u = \Delta(\Delta u)$ is the bi-Laplacian.

The dynamics of the $Q-u$ SmA model are governed by a system of highly nonlinear, strongly coupled gradient-flow equations. The variational derivation of the $Q-u$ free energy results in a stiff, multi-scale system: a second-order nonlinear parabolic partial differential equation (PDE) for the tensor $Q$, coupled with a typically fourth-order highly nonlinear parabolic PDE for the layer variable $u$. Solving this coupled system numerically presents severe computational bottlenecks. The stiff linear differential operators demand prohibitively small time steps for stability if treated explicitly. Conversely, standard fully implicit treatments necessitate the solution of massive, fully coupled, nonlinear algebraic systems at each time step, which is computationally expensive and often struggles with convergence, especially in three-dimensional simulations \cite{shi2025modified}.

To address linear stiffness in gradient flows, Exponential Time Differencing (ETD) integrators exactly resolve the linear part, allowing for significantly larger time steps while maintaining high accuracy, making them highly successful in many classical phase-field models such as the Allen-Cahn and Cahn-Hilliard equations.
However, the direct application of ETD to highly nonlinear systems does not naturally guarantee the decay of the original free energy, which is a fundamental thermodynamic property of gradient flows \cite{du2018stabilized,du2019,du2021,liu2025maximum}.

In recent years, the Scalar Auxiliary Variable (SAV) approach has revolutionized the design of structure-preserving algorithms by introducing a scalar variable for the nonlinear free energy, thereby allowing the nonlinear terms to be treated explicitly or ly. This leads to numerical schemes that are linear, decoupled and, most importantly, unconditionally energy-stable with respect to a modified discrete energy.
Leveraging the synergy of ETD and SAV, [4] proposed stabilized linear schemes up to second-order for Allen-Cahn flows. Their approach uniquely achieves unconditional preservation of both energy dissipation and the maximum bound principle (MBP)-overcoming typical time-step restrictions for higher-order methods-while providing optimal error estimates \cite{shen2019new,shen2018scalar,zhang2022generalized,jiang2022improving,liu2024novel,zhang2025novel}.

The integration of ETD and SAV methodologies has recently been successfully applied to various complex systems. For instance, stabilized Exponential-SAV (sESAV) schemes have been developed for scalar Allen-Cahn-type gradient flows, demonstrating unconditional modified-energy decay and preservation of the maximum bound principle (MBP). More directly related to our work, sESAV schemes have been successfully extended to nematic $Q$-tensor flows, proving that the multi-component tensor structure and physical constraints of liquid crystals can be elegantly managed within this framework.
Furthermore, for smectic layer dynamics (without the $Q$-tensor coupling), unconditionally energy-stable schemes based on auxiliary variable methods, such as the Invariant Energy Quadratization (IEQ) and BDF2-SAV methods, have been developed, confirming the efficacy of SAV-type approaches for handling the severe nonlinearities inherent in high-order smectic layer energies \cite{ju2022generalized,hou2025energy}.

Despite these significant advancements, there is currently no work that directly applies SAV-EI scheme to the comprehensive $Q-u$ Smectic-A gradient-flow model. Existing numerical studies on the $Q-u$ framework have primarily focused on static minimization problems or employed standard integration techniques that may suffer from severe time-step restrictions or nonlinear solver complexities.

In this paper, we bridge this gap by proposing and analyzing a novel, efficient, and rigorously energy-stable numerical scheme for the $Q-u$ Smectic-A gradient-flow model. We develop a fully discrete scheme by combining the Exponential Time Differencing (ETD) method for the stiff linear spatial operators with the generalized Scalar Auxiliary Variable (SAV) approach for the highly nonlinear coupling terms. This combination allows us to construct a scheme that is linear and decoupled at each time step, requiring only the solution of constant-coefficient linear systems, which can be solved extremely fast (e.g., via Fast Fourier Transform). Most critically, we rigorously prove that our proposed EOP-GSAV-EI  scheme is unconditionally energy-stable, ensuring thermodynamic consistency and long-term numerical robustness even for large time steps.

The rest of the paper is organized as follows. In Section 2, we present the formulation of the fully discrete EOP-GSAV-EI  scheme for the $Q-u$ Smectic-A model. We rigorously prove the unconditional energy stability of the proposed scheme in Section 3. In Section 4, we provide comprehensive numerical experiments to demonstrate the accuracy, efficiency, and structural preservation properties of our algorithm, as well as its capability in capturing complex defect dynamics and phase transitions. Finally, we conclude with a summary and discussion of future research directions in Section 5.
\section{ Fully discrete EOP-GSAV-EI finite difference scheme}
We first split the total free energy $E(\mathbf{Q}, u)$ into a linear quadratic part and a nonlinear part:
\begin{equation}
    E(\mathbf{Q}, u) = E_0(\mathbf{Q}, u) + E_1(\mathbf{Q}, u),
\end{equation}
where
\begin{align*}
    E_0(\mathbf{Q}, u) &=  \int_{\Omega} \left( K |\nabla \mathbf{Q}|^2 + 2B_0 |\Delta u|^2 \right) d\mathbf{x}, \\
    E_1(\mathbf{Q}, u) &= \int_{\Omega} 2 B_0(D^2 u : q^2 M u )+B_0 \left| q^2 M u \right|^2 + f_{bn}(\mathbf{Q}) + f_{s}(u) \, d\mathbf{x}.
\end{align*}

From the article of [4], we know that the regularity of the solution and the structure of the nonlinear energy ensure that $u \in H^2(\Omega)$ and $\mathbf{Q} \in H^1(\Omega)$ are uniformly bounded in time.
Thus, there exist two positive constants $C_*, C^*$ such that
\begin{equation}
    -C_* \le E_1(\mathbf{Q}(t), u(t)) \le C^*,
\end{equation}
where  we have used the fact that
\begin{align*}
     \int_{\Omega} 2 B_0(D^2 u : q^2 M u ) \, d\mathbf{x} \leq  \int_{\Omega} \left( 2B_0 \cdot \left| q^2 M u \right|^2 + 2B_0 |D^2 u|^2 \right) d\mathbf{x}\leq  C^*.
\end{align*}
By introducing a scalar auxiliary variable $s(t) = E_1(\mathbf{Q}, u)$, we formulate the modified total energy $\mathcal{E}$ and an exponential stabilization factor $g$ to ensure unconditional discrete energy dissipation:
\begin{align}
    g(\mathbf{Q}, u, s) := \frac{\exp(s)}{\exp(E_1(\mathbf{Q}, u))}, \quad
    \mathcal{E}(\mathbf{Q}, u, s) := \frac{L_1}{2}\|\nabla \mathbf{Q}\|^2 + \frac{B_0}{2}\|\Delta u\|^2 + s \ge -C_*. \label{eq_modified_system}
\end{align}
At the continuous level, $g \equiv 1$, and $\mathcal{E}$ exactly recovers the original free energy $E(\mathbf{Q}, u)$.

Denoting the nonlinear variations as $\mathbf{H} = \frac{\delta E_1}{\delta \mathbf{Q}}$, $\mu = \frac{\delta E_1}{\delta u}$, the governing equations for $\mathbf{Q}$ and $u$ are coupled with the evolution of $s(t)$ to form the complete equivalent system:
\begin{align}
    \frac{\partial \mathbf{Q}}{\partial t} -K\Delta\mathbf{Q} &= -  g(\mathbf{Q}, u, s) \mathbf{H}, \label{eq_sys_Q} \\
    \frac{\partial u}{\partial t} + 2B_0\Delta^2 u &= -  g(\mathbf{Q}, u, s) \mu, \label{eq_sys_u} \\
    \frac{d s}{dt} &= g \left( \int_{\Omega} \mathbf{H} : \frac{\partial \mathbf{Q}}{\partial t} \, d\mathbf{x} + \int_{\Omega} \mu \frac{\partial u}{\partial t} \, d\mathbf{x} \right). \label{eq_sys_s}
\end{align}
\subsection{Spatial Discretization and Discrete Function Spaces}~\\
To simplify the notation, we consider the problem in the domain $\Omega = [0, L_d]^3$. Given a positive integer $J$, the uniform mesh partitioning size for each spatial direction is set to be $h = L_d/J$. For a regular grid, all variables are stored at the primary mesh points. We denote by $\mathbf{E}$ the set of mesh points, defined by
\begin{equation*}
    \mathbf{E} = \left\{ (x_p, y_q, z_r) = (ph, qh, rh) \mid p, q, r = 0, 1, \dots, J \right\}.
\end{equation*}
The corresponding periodic grid function space is given by
\begin{equation*}
    E_h^{\text{per}} = \{ U : \mathbf{E} \to \mathbb{R} \mid U_{0,q,r} = U_{J,q,r}, \, U_{p,0,r} = U_{p,J,r}, \, U_{p,q,0} = U_{p,q,J} \text{ for all } 0 \le p, q, r \le J \}.
\end{equation*}
For ease of notation, we handle the periodic boundary conditions by assigning the ghost nodes the values of their corresponding interior nodes; that is, for any $0 \le p, q, r \le J$,
\begin{equation*}
    U_{-1,q,r} = U_{J-1,q,r}, \quad U_{J+1,q,r} = U_{1,q,r},
\end{equation*}
and analogous periodic extensions apply to the $q$ and $r$ directions.

Then, we define the forward, backward, and central difference operators for the grid function $U \in E_h^{\text{per}}$ along the $x$-direction as:
\begin{align*}
    (D_1^+ U)_{p,q,r} = \frac{U_{p+1,q,r} - U_{p,q,r}}{h},\quad
    (D_1^- U)_{p,q,r} = \frac{U_{p,q,r} - U_{p-1,q,r}}{h},\quad
    (D_1^c U)_{p,q,r} = \frac{U_{p+1,q,r} - U_{p-1,q,r}}{2h},
\end{align*}
where the difference operators $D_2^{\pm,c}$ and $D_3^{\pm,c}$ along the $y$- and $z$-directions are defined analogously. Moreover, we denote the mixed central difference as $D_{k,l}^c U = D_k^c D_l^c U$ for $k,l \in \{1,2,3\}$ for simplicity.

In a collocated grid, to maintain a compact stencil for the Laplacian, we define the discrete gradient operator $\nabla_h : E_h^{\text{per}} \to (E_h^{\text{per}})^3$ using forward differences:
\begin{equation*}
    (\nabla_h U)_{p,q,r} = ( (D_1^+ U)_{p,q,r}, (D_2^+ U)_{p,q,r}, (D_3^+ U)_{p,q,r} )^T,
\end{equation*}
and the corresponding discrete divergence operator $\nabla_h \cdot : (E_h^{\text{per}})^3 \to E_h^{\text{per}}$ using backward differences:
\begin{equation*}
    \nabla_h \cdot (U^{(1)}, U^{(2)}, U^{(3)})^T = D_1^- U^{(1)} + D_2^- U^{(2)} + D_3^- U^{(3)}.
\end{equation*}

Because all variables are collocated, the discrete inner products no longer require averaging operators. They are simply given by:
\begin{align*}
    \langle U, V \rangle_h = h^3 \sum_{p,q,r=1}^{J-1} U_{p,q,r} V_{p,q,r} \quad \forall U, V \in E_h^{\text{per}},\quad
    [\mathbf{U}, \mathbf{V}]_h = \sum_{k=1}^3 \langle U^{(k)}, V^{(k)} \rangle_h \quad \forall \mathbf{U}, \mathbf{V} \in (E_h^{\text{per}})^3.
\end{align*}
Then, for any $U \in E_h^{\text{per}}$, the corresponding discrete $l^2$, $H^1$, $H^2$, and $\infty$ norms are given by
\begin{align*}
    \|U\|_h^2 &= \langle U, U \rangle_h, \quad
    \|\nabla_h U\|_h^2 = [\nabla_h U, \nabla_h U]_h = \sum_{k=1}^3 \langle D_k^+ U, D_k^+ U \rangle_h, \\
    \|U\|_{H_h^1}^2 &= \|U\|_h^2 + \|\nabla_h U\|_h^2, \quad
    \|U\|_{H_h^2}^2 = \|U\|_{H_h^1}^2 + \|\Delta_h U\|_h^2,\quad
    \|U\|_\infty = \max_{0 \le p,q,r \le J} |U_{p,q,r}|,
\end{align*}
where the discrete Laplacian is defined as $\Delta_h U = \sum_{k=1}^3 D_k^+ D_k^- U$, and the discrete biharmonic operator is defined as $\Delta_h^2 U = \Delta_h (\Delta_h U)$.
Using summation-by-parts on the collocated grid, it is easy to check that, for any $U, V \in E_h^{\text{per}}$, it holds that
\begin{equation*}
\begin{aligned}
    \langle D_k^- D_k^+ U, V \rangle_h &= - \langle D_k^+ U, D_k^+ V \rangle_h = \langle U, D_k^- D_k^+ V \rangle_h, \\
    \langle D_{k,l}^c U, V \rangle_h &= - \langle D_l^c U, D_k^c V \rangle_h = - \langle D_k^c U, D_l^c V \rangle_h = \langle U, D_{k,l}^c V \rangle_h,
\end{aligned}
\end{equation*}
which also implies that
\begin{equation*}
    \langle \Delta_h U, V \rangle_h = - [\nabla_h U, \nabla_h V]_h = \langle U, \Delta_h V \rangle_h.
\end{equation*}

For the tensor-valued grid functions, we define the discrete function space $ \mathbf{S}_h^{(3)}$ for the $Q$-tensor as
\begin{equation*}
    \mathbf{S}_h^{(3)} = \left\{ \Phi_h \bigg| \Phi_h^{i,j} \in E_h^{\text{per}}, \, \operatorname{tr}(\Phi_h) := \sum_{i=1}^3 \Phi_h^{ii} = 0, \, \Phi_h^{i,j} = \Phi_h^{j,i}, \, i,j=1,2,3 \right\}.
\end{equation*}
The discrete $l^2$ norm, $l^\infty$ norm, $H_h^1$ norm and $H_h^2$ norm of the grid tensor-functions $\Phi_h, \Psi_h \in \mathbf{S}_h^{(3)}$ are defined respectively by
\ba\label{eq_norm}
 \|\Phi_h\|_h := \sqrt{\sum_{i,j=1}^3 \|\Phi_h^{i,j}\|_h^2}, \quad
    \|\Phi_h\|_\infty := \max_{0 \le p,q,r \le J} |\Phi_{h}|_F,\quad    |\Phi_{h}|_F := \left( \sum_{i,j=1}^3 |\Phi_{h}^{i,j}|^2 \right)^{1/2},\\    \|\nabla_h \Psi_h\|_h := \sqrt{\sum_{k=1}^3 \sum_{i,j=1}^3 \|D_k^+ \Psi_h^{i,j}\|_h^2}, \quad
    \|\Psi_h\|_{H_h^1} := \sqrt{\|\Psi_h\|_h^2 + \|\nabla_h \Psi_h\|_h^2},\\
    \|\Delta_h \Psi_h\|_h := \sqrt{\sum_{i,j=1}^3 \|\Delta_h \Psi_h^{i,j}\|_h^2}, \quad
    \|\Psi_h\|_{H_h^2} := \sqrt{\|\Psi_h\|_{H_h^1}^2 + \|\Delta_h \Psi_h\|_h^2}.
\ed
Moreover, the tensor-matrix product and discrete Frobenius product for any $\Phi_h, \Psi_h \in \mathbf{S}_h^{(3)}$ are simplified as all components are collocated:
\begin{equation*}
    (\Phi_h \Psi_h)^{i,j}_{p,q,r} := \sum_{k=1}^3 (\Phi_h^{i,k})_{p,q,r} (\Psi_h^{k,j})_{p,q,r}, \quad
    \langle \Phi_h, \Psi_h \rangle_h = \sum_{i,j=1}^3 \langle \Phi_h^{i,j}, \Psi_h^{i,j} \rangle_h.
\end{equation*}
\subsection{Fully discrete scheme}~\\
Let $\tau$ be the time step size. Set $t_{n+1} = t_n + \tau$, and denote the numerical approximations of $\mathbf{Q}_h(t_n)$, $u_h(t_n)$, and $s(t_n)$ by $\mathbf{Q}_h^n$, $u_h^n$, and $s^n$, respectively.
We first define the fully discrete approximation of the continuous free energy $E_1$ at time level $t_n$, denoted as $E_{1h}^n$ as follows:
\begin{equation}\label{eq_s_ex_n}
\begin{aligned}
    E_{1h}^n
    = 2 B_0 q^2 \langle D_h^2 u_h^n , M_h^n u_h^n \rangle_h + B_0 q^4 \| M_h^n u_h^n \|_h^2 + \langle f_{bn}(\mathbf{Q}_h^n), 1 \rangle_h + \langle f_{s}(u_h^n), 1 \rangle_h,
\end{aligned}
\end{equation}
where $M_h^n = \frac{\mathbf{Q}_h^n}{s_+} + \frac{\mathbf{I}_d}{d}$ is the discrete evaluation of the tensor $M$ at the current state.

Subsequently, to account for the discrepancy between the numerical auxiliary variable and the exact physical energy, we define the scalar relaxation factor $g^n$ as:
\begin{equation}\label{eq:g_n}
    g^n := g(\mathbf{Q}_h^n, u_h^n, s^n) = \frac{e^{s^n}}{e^{E_{1h}^n}}.
\end{equation}
Then, we define the modified discrete linear operators $\mLQ: \mathbf{S}_h^{(d)} \to \mathbf{S}_h^{(d)}$ and $\mLu: E_h^{\text{per}} \to E_h^{\text{per}}$ with the positive stabilizing parameters $\kappa_1$ and $\kappa_2$ as follows:
\begin{align} \label{eq_L}
    \mLQ\mathbf{Q}_h^n = -K\Delta_h \mathbf{Q}_h^n + g^n \kappa_1 \mathbf{Q}_h^n,\quad
    \mLu u_h^n = 2B_0\Delta_h^2 u_h^n + g^n \kappa_2 u_h^n.
\end{align}
It is clear that the operators $\mLQ$ and $\mLu$ are symmetric positive-definite, applying the summation-by-parts formula naturally induces their corresponding weighted discrete energy norms:
$ \ilQ{ \cdot }  := K  \ilQ{\nabla_h \cdot } + g^n \kappa_1  \ilQ{\cdot}$
and
$\ilu{\cdot} := 2B_0  \ilu{\Delta_h \cdot } + g^n \kappa_2 \ilu{\cdot}$.

Then, we define the nonlinear terms $\mathcal{N}_{\mathbf{Q},h}^n$ and $\mathcal{N}_{h}^n$ with the positive stabilizing parameters $\kappa_1$ and $\kappa_2$ as the discrete approximations of the nonlinear variations in the modified system \eqref{eq_modified_system}:
\ba \label{eq_N_Q_def}
\mNQz(t_n) &\coloneqq  \mNQz(\mathbf{Q}(t_n), u(t_n))= -g(t_n) \mathbf{H}(t_n)+ g(t_n)\kappa_1 \mathbf{Q}(t_n),\\
\mathbf{N}_{h}^n &:= \mathbf{N}_{h}(\mathbf{Q}_h^n, u_h^n, s^n) = -g^n \muQ^n + g^n \kappa_1 \mathbf{Q}_h^n, \\
\mNuz(t_n) &\coloneqq  \mNuz(\mathbf{Q}(t_n), u(t_n))= -g(t_n) \mu(t_n)+ g(t_n)\kappa_2 u(t_n),\\
\mathcal{N}_{h}^n &:= \mathcal{N}_{h}(\mathbf{Q}_h^n, u_h^n, s^n) = -g^n \mu_{h}^n + g^n \kappa_2 u_h^n.
\ed
where $\muQ^n$ and $\mu_{h}^n$ are the discrete approximations of the nonlinear variations $\mathbf{H}$ and $\mu$ at time level $t_n$, respectively, given by
\begin{equation}\label{eq_H_mu_def}
\begin{aligned}
    \muQ^n &= A \mathbf{Q}_h^n - B \left( (\mathbf{Q}_h^n)^2 - \frac{\mathrm{tr}((\mathbf{Q}_h^n)^2)}{3}\mathbf{I} \right) + C \mathrm{tr}((\mathbf{Q}_h^n)^2)\mathbf{Q}_h^n \\&\quad + \frac{2B_0 q^2}{s_+} \left( u_h^n D_h^2 u_h^n - \frac{\mathrm{tr}(u_h^n D_h^2 u_h^n)}{3}\mathbf{I} \right) + 2B_0 q^4 \frac{\mathbf{Q}_h^n}{s_+^2} (u_h^n)^2,\\\mu_{h}^n = au_h^n + b(u_h^n)^2 &+ c(u_h^n)^3 + 2B_0 q^2 M_h^n : D_h^2 u_h^n  + 2B_0 q^2 \nabla_h \cdot \big( \nabla_h \cdot (M_h^n u_h^n) \big) + 2B_0 q^4 |M_h^n|^2 u_h^n.
\end{aligned}
\end{equation}

Applying the variation-of-constants formula and approximating the nonlinear terms by their values at $t_n$, we obtain the following first-order fully discrete EOP-GSAV-EI scheme for the $Q-u$ Smectic-A model:
\begin{align}
    \bQ^{n+1} &= e^{-\tau \mLQ} \bQ^{n} + \tau \phi_1(-\tau \mLQ) \mNQ^n, \label{eq4_6}\\
    \bu^{n+1} &= e^{-\tau \mLu} \bu^{n} + \tau \phi_1(-\tau \mLu) \mNu^n,\label{eq4-7}\\
    \tilde{s}^{n+1} &= s^n + g^n \left( \left\langle \muQ^n, \bQ^{n+1} - \bQ^{n} \right\rangle_h + \left\langle \mu_{h}^n, \bu^{n+1} - \bu^{n} \right\rangle_h \right),\label{eq4-8}
\end{align}
where $\phi_1(z) = (e^z - 1)/z$, and $\tilde{s}^{n+1}$ is the provisional auxiliary variable.

 To enhance the stability, we can reformulate the above scheme into a  form by multiplying both sides of the equations for $\bQ^{n+1}$ and $\bu^{n+1}$ by $\frac{1}{\tau}(\mq(\tau \mLQ)+\tau \mLQ)$ and $\frac{1}{\tau}(\mq(\tau \mLu)+\tau \mLu)$, respectively, where $\mq(z) = z / (e^z - 1)$ is the inverse of $\phi_1(z)$.
we first obtain the following equivalent
\begin{align*}
    \frac{1}{\tau}(\mq(\tau \mLQ)+\tau \mLQ)=& \frac{1}{\tau} \left( \frac{\tau \mLQ}{e^{\tau \mLQ} - 1} + \tau \mLQ \right) = \frac{1}{\tau} \left( \frac{\tau \mLQ e^{\tau \mLQ}}{e^{\tau \mLQ} - 1} \right) .
\end{align*}
Then we can get the following key identity for the linear part:
\begin{align*}
    \frac{1}{\tau} \left( \frac{\tau \mLQ e^{\tau \mLQ}}{e^{\tau \mLQ} - 1} \right)*\tau \phi_1(-\tau \mLQ) =& \frac{1}{\tau} \left( \frac{\tau \mLQ e^{\tau \mLQ}}{e^{\tau \mLQ} - 1} \right) * \tau \left( \frac{e^{-\tau \mLQ} - 1}{-\tau \mLQ} \right) \\=& \frac{1}{\tau} \left( \frac{\tau \mLQ e^{\tau \mLQ}}{e^{\tau \mLQ} - 1} \right) * \tau \left( \frac{1 - e^{-\tau \mLQ}}{\tau \mLQ} \right) = 1.
\end{align*}
Then we have for the $\bQ^{n+1} - e^{-\tau \mLQ} \bQ^{n}$ part:
\begin{align*}
    \frac{1}{\tau}(\mq(\tau \mLQ)+\tau \mLQ) * (\bQ^{n+1} - e^{-\tau \mLQ} \bQ^{n})&=\frac{1}{\tau}\mq(\tau \mLQ)*\bQ^{n+1}+\tau \mLQ * \bQ^{n+1}-\frac{1}{\tau}\mq(\tau \mLQ)* \bQ^{n}\\
    &= \mq(\tau \mLQ) \frac{\bQ^{n+1} - \bQ^n}{\tau} + \mLQ \bQ^{n+1}.
\end{align*}
Then the above scheme can be equivalently reformulated as the following  scheme:
\begin{align}
    \mq(\tau \mLQ) \frac{\bQ^{n+1} - \bQ^n}{\tau} + \mLQ \bQ^{n+1} = \mNQ^n, \label{eq_scheme_Q}\\
    \mq(\tau \mLu) \frac{\bu^{n+1} - \bu^{n}}{\tau} + \mLu \bu^{n+1} = \mNu^n, \label{eq_scheme_u}\\
    \ts^{n+1} = s^n + g^n \left( \left\langle \muQ^n, \delta \bQ^{n+1} \right\rangle_h + \left\langle \mu_{h}^n, \delta \bu^{n+1} \right\rangle_h \right), \label{eq_scheme_s}
\end{align}
where $\mq(z) = z / ( e^{z} - 1)$, $\delta \bQ^{n+1} = \bQ^{n+1} - \bQ^{n}$, and $\delta \bu^{n+1} = \bu^{n+1} - \bu^{n}$.

To eliminate the long-term truncation error in $\tilde{s}^{n+1}$ while preserving unconditional energy stability, we introduce a continuous relaxation step (Relaxed SAV):
\begin{equation}\label{eq4-9}
    s^{n+1} = \xi^{n+1} \tilde{s}^{n+1} + (1 - \xi^{n+1}) E_{1h}^{n+1}, \quad \xi^{n+1} \in \mathcal{V}.
\end{equation}
The feasible set $\mathcal{V}$ is defined as $\mathcal{V} = \mathcal{V}_1 \cap \mathcal{V}_2$, with
\begin{align*}
    \mathcal{V}_1 &= \{ \xi \mid \xi \in [0, 1] \}, \label{eq4-10a}\\
    \mathcal{V}_2 &= \left\{ \xi \mid s^{n+1} - \tilde{s}^{n+1} \leqslant  \eta_0 \tau \mathcal{R}^{n+1}, \ s^{n+1} = \xi \tilde{s}^{n+1} + (1 - \xi) E_{1h}^{n+1} \right\}.
\end{align*}
Here, $\eta_0 \in [0, 1]$ is an artificial parameter manually assigned to reserve a portion of the energy dissipation, and $\mathcal{R}^{n+1}=\frac{1}{\tau}(\ibQ{\delta \bQ^{n+1}}+\ibu{\delta \bu^{n+1}})\geq 0 $ is the discrete energy dissipation rate.
\begin{remark}
To minimize the truncation error (i.e., making $s^{n+1}$ as close to $E_{1h}^{n+1}$ as possible, which implies minimizing $\xi$), the optimal choice for the relaxation parameter $\xi^{n+1}$ is determined by solving the following constrained optimization problem:
\begin{equation}
    \xi^{n+1} = \min \{ \xi \mid \xi \in [0, 1], \ f(\xi) \leqslant 0 \},
\end{equation}
where $f(\xi)$ is a linear function derived from the discrete energy dissipation law, given by:
\begin{equation}
    f(\xi) = (1 - \xi)(E_{1h}^{n+1} - \tilde{s}^{n+1}) - \eta_0 \tau \mathcal{R}^{n+1}.
\end{equation}

Since $f(\xi)$ is linear with respect to $\xi$, the optimal relaxation parameter $\xi^{n+1}$ can be explicitly and efficiently determined without solving a quadratic equation. Specifically, we have:
\begin{equation}\label{eq_xi_optimal}
    \xi^{n+1} =
    \begin{cases}
        0, & \text{if } E_{1h}^{n+1} \leqslant \tilde{s}^{n+1}, \\
        \max \left( 0, 1 - \frac{\eta_0 \tau \mathcal{R}^{n+1}}{E_{1h}^{n+1} - \tilde{s}^{n+1}} \right), & \text{if } E_{1h}^{n+1} > \tilde{s}^{n+1}.
    \end{cases}
\end{equation}
\end{remark}
\begin{remark}
The standard implicit backward Euler scheme is renowned for its well-established and elegant theoretical frameworks, particularly concerning the construction of appropriate test functions for error estimates and the rigorous derivation of energy dissipation laws. By employing the aforementioned system reformulation, our proposed scheme seamlessly inherits these analytical advantages. The profound mathematical benefits of this transformation will become persistently evident throughout the subsequent rigorous proofs.
\end{remark}
\section{Energy Dissipation Analysis}~\\
To prepare for the stability analysis, we introduce the modified discrete energy functionals and norms, along with a key lemma regarding their bounds and equivalences.

The total modified discrete energy functionals are given by
\ba\label{energy_discrete}
    \widetilde{\mathcal{E}}_h(\mathbf{Q}_h, u_h, \tilde{s}) &= \frac{K}{2} \|\nabla_h \mathbf{Q}_h\|_h^2  + B_0 \left\| \Delta_h u_h  \right\|_h^2+\tilde{s},\\
    \mathcal{E}_h(\mathbf{Q}_h, u_h,s) &= \frac{K}{2} \|\nabla_h \mathbf{Q}_h\|_h^2  + B_0 \left\| \Delta_h u_h  \right\|_h^2+s.
\ed
Since the operators $\mLQ$ and $\mLu$ are symmetric positive-definite and the functions $\mathcal{Q}(z)$, and $\mathcal{Q}_1(z)$ are non-negative for all $z \ge 0$, the resulting operators $\mathcal{Q}(\mLQ)$,  and $\mathcal{Q}_1(\mLQ)$ are also symmetric positive-definite. Based on this property, for a given time step size $\tau > 0$, we construct the following modified discrete norms for any $\mathbf{U} \in E_h^{\text{per}}$:
\ba\label{eq_inner_Q}
     \iaQ{\mathbf{U}} &:= \langle \mathcal{Q}(\tau \mathcal{L}_h) \mathbf{U}, \mathbf{U} \rangle_h,  \\
    \ibQ{\mathbf{U}} &:= \langle \mathcal{Q}_1(\tau \mathcal{L}_h) \mathbf{U}, \mathbf{U} \rangle_h = \langle \mathcal{Q}(\tau \mathcal{L}_h) \mathbf{U}, \mathbf{U} \rangle_h + \frac{\tau}{2} \langle \mathcal{L}_h \mathbf{U}, \mathbf{U} \rangle_h, \\
    \icQ{\mathbf{U}} &:=  \langle \mathcal{Q}(\tau \mathcal{L}_h)  \mathcal{L}_h   \mathbf{U}, \mathbf{U} \rangle_h,
\ed
where the functions $\mathcal{Q}(z)$,  and $\mathcal{Q}_1(z)$ are defined as follows:
\begin{equation}\label{eq_Q_functions}
    \mathcal{Q}(z) = \frac{z}{e^z - 1}\geq 0, \quad \mathcal{Q}_1(z) = \frac{z}{e^z - 1}+ \frac{z}{2}\geq 0, \quad \text{for}\  z\geq 0.
\end{equation}
The same definitions are extended to the grid functions $V \in E_h^{\text{per}}$ by incorporating the operator $\mathcal{D}_h$, yielding the norms denoted by $ \iau{V} =\langle \mathcal{Q}(\tau \mathcal{D}_h) V, V \rangle_h
$ and $\ibu{ V}=\langle \mathcal{Q}_1(\tau \mathcal{D}_h) V, V \rangle_h$.
\begin{lemma}\label{lem_norm_equivalence}
 For any grid function $\mathbf{U} \in \mathbf{S}_h^{(d)},V \in E_h^{\text{per}}$, the modified discrete norms satisfy the following bounds and equivalences:
\ba\label{eq_norm_bounds_1}
    \|\mathcal{Q}(\tau \mathcal{L}_h) \mathbf{U}\|_h^2 &\le \|\mathbf{U}\|_{\mathcal{Q}_\mathcal{L}}^2 \le \|\mathbf{U}\|_h^2 \le \|\mathbf{U}\|_{\mathcal{Q}_\mathcal{L}^1}^2,\\
    0 &\le \langle \mathcal{Q}(\tau \mathcal{L}_h)\mathcal{L}_h \mathbf{U}, \mathbf{U} \rangle_h \le  \langle \mathcal{L}_h \mathbf{U}, \mathbf{U} \rangle_h= \ilQ{\mathbf{U}},\\
      0 &\le \ibQ{\mathbf{U}} \le  \langle \mathcal{L}_h \mathbf{U}, \mathbf{U} \rangle_h= \ilQ{\mathbf{U}},\\
    \|\mathcal{Q}(\tau \mathcal{D}_h) V\|_h^2 &\le \|V\|_{\mathcal{Q}_\mathcal{D}}^2 \le \|V\|_h^2 \le \|V\|_{\mathcal{Q}_\mathcal{D}^1}^2,\\
    0 &\le \langle \mathcal{Q}(\tau \mathcal{D}_h)\mathcal{D}_h V, V \rangle_h \le  \langle \mathcal{D}_h V, V \rangle_h= \ilu{V},\\
      0 &\le \ibu{V} \le  \langle \mathcal{D}_h V, V \rangle_h= \ilu{V}.
\ed
and the same bounds and equivalences hold for $U \in E_h^{\text{per}}$ with the norms $\| \cdot \|_{\mathcal{Q}_\mathcal{D}}$ and $\| \cdot \|_{\mathcal{Q}_\mathcal{D}^1}$.
\end{lemma}
\begin{proof}
Since $\mathcal{L}_h, \mathcal{D}_h$ are  symmetric positive semi-definite operators, all its discrete eigenvalues $z$ are non-negative. By analyzing the scalar generating functions defined in \eqref{eq_Q_functions}, it is straightforward to verify that for all $z \ge 0$:
\begin{align*}
    \mathcal{Q}(z) \le 1,\quad \mathcal{Q}_1(z) = \frac{z}{e^z - 1}+ \frac{z}{2}\geq 1.
\end{align*}
Then we can derive the following inequalities for the scalar functions $\mathcal{Q}(z)$,  and $\mathcal{Q}_1(z)$:
\begin{equation*}
    0 \le \mathcal{Q}^2(z) \le \mathcal{Q}(z) \le 1 \le \mathcal{Q}_1(z), \quad \text{and} \quad 0 \le \mathcal{Q}(z)z \le z.
\end{equation*}
Applying the spectral mapping theorem to the operator $\tau\mathcal{L}_h$ and $\tau\mathcal{D}_h$, these scalar inequalities naturally extend to the corresponding operator inequalities. For any $\mathbf{U} \in \mathbf{S}_h^{(d)}$, we have
\begin{align*}
 \ibQ{\mathbf{U}} &= \langle \mathcal{Q}(\tau \mathcal{L}_h) \mathbf{U}, \mathbf{U} \rangle_h + \frac{\tau}{2} \langle \mathcal{L}_h \mathbf{U}, \mathbf{U} \rangle_h \\
& \leq \langle  \mathbf{U}, \mathbf{U} \rangle_h + \frac{\tau}{2} \langle \mathcal{L}_h \mathbf{U}, \mathbf{U} \rangle_h\leq \frac{1}{2}g^n \kappa_1 \langle  \mathbf{U}, \mathbf{U} \rangle_h +  \frac{1}{2}\langle \mathcal{L}_h \mathbf{U}, \mathbf{U} \rangle_h\leq \langle \mathcal{L}_h \mathbf{U}, \mathbf{U} \rangle_h.
\end{align*}
And for any $V \in E_h^{\text{per}}$, we have
\begin{align*}
    \ibu{V} &= \langle \mathcal{Q}(\tau \mathcal{D}_h) V, V \rangle_h + \frac{\tau}{2} \langle \mathcal{D}_h V, V \rangle_h \\
    & \leq \langle  V, V \rangle_h + \frac{\tau}{2} \langle \mathcal{D}_h V, V \rangle_h\leq \frac{1}{2}g^n \kappa_2 \langle  V, V \rangle_h +  \frac{1}{2}\langle \mathcal{D}_h V, V \rangle_h\leq \langle \mathcal{D}_h V, V \rangle_h.
\end{align*}
\end{proof}
\begin{theorem}[Unconditional Energy Stability]\label{them3_1}
The proposed first-order  EOP-GSAV-EI  scheme with dual stabilization is unconditionally energy dissipative. That is, for any time step size $\tau > 0$, and stabilization parameters $\kappa_1, \kappa_2 > 0$,  the discrete modified energy satisfies the following dissipation law:
\ba
    \mathcal{E}_h(\bQ^{n+1}, \bu^{n+1}, s^{n+1}) \le \mathcal{E}_h(\bQ^{n}, \bu^{n}, s^n),\quad \forall n \ge 0,
\ed
where $\mathcal{E}_h$ is defined as  \eqref{energy_discrete}.
\end{theorem}

\begin{proof}
Taking the inner product of equations (\ref{eq_scheme_Q}) and (\ref{eq_scheme_u}) with $\delta \bQ^{n+1}$ and $\delta \bu^{n+1}$, respectively, we obtain:
\ba \label{eq34}
    \langle \mLQ\bQ^{n+1}, \delta \bQ^{n+1} \rangle_h &=-\frac{1}{\tau}\langle \mq(\tau \mLQ) \delta \bQ^{n+1}, \delta \bQ^{n+1} \rangle_h+\id{\mNQ^n}{\delta \bQ^{n+1}}
    \\&= -\frac{1}{\tau}\iaQ{\delta \bQ^{n+1}} + g^n \kappa_1 \langle \mathbf{Q}_h^n, \delta \bQ^{n+1} \rangle_h - g^n \langle \muQ^n, \delta \bQ^{n+1} \rangle_h, \\
    \langle \mLu\bu^{n+1}, \delta \bu^{n+1} \rangle_h &= -\frac{1}{\tau}\langle \mq(\tau \mLu) \delta \bu^{n+1}, \delta \bu^{n+1} \rangle_h + \langle \mNu^n, \delta \bu^{n+1} \rangle_h \\
    &=-\frac{1}{\tau}\iau{\delta \bu^{n+1}} + g^n \kappa_2 \langle u_h^n, \delta \bu^{n+1} \rangle_h - g^n \langle \mu_h^n, \delta \bu^{n+1} \rangle_h,
\ed
Then using the standard algebraic identity, we evaluate the difference between the provisional modified energy $\widetilde{\mathcal{E}}^{n+1}$ and the energy $\mathcal{E}^n$:
\begin{align}
    \widetilde{\mathcal{E}}^{n+1} - {\mathcal{E}}^n &= \frac{1}{2}\ia{\nabla_h  \bQ^{n+1} }- \frac{1}{2}\ia{\nabla_h \bQ^{n}} + \frac{1}{2}\ia{\Delta_h\bu^{n+1}} - \frac{1}{2}\ia{\Delta_h\bu^{n}}  + \widetilde{s}^{n+1} - {s}^n \no\\
    &=  \id{\mLQ\bQ^{n+1}}{\delta \bQ^{n+1}}  - \frac{1}{2}\|\nabla_h \delta \bQ^{n+1}\|_h^2 -g^n \kappa_1 \id{ \mathbf{Q}_h^{n+1}}{ \delta \bQ^{n+1}} + \langle \mLu\bu^{n+1}, \delta \bu^{n+1} \rangle_h \no\\
    &- \frac{1}{2} \ia{\Delta_h \delta \bu^{n+1}} - g^n \kappa_2 \langle u_h^{n+1}, \delta \bu^{n+1} \rangle_h + g^n \left( \langle \muQ^n, \delta \bQ^{n+1} \rangle_h + \langle \mu_h^n, \delta \bu^{n+1} \rangle_h \right).\label{eq_diff_E}
\end{align}
Substituting \eqref{eq34} back into (\ref{eq_diff_E}) and using Lemma~\ref{lem_norm_equivalence}, we arrive at the following key inequality:
\begin{equation}
\begin{aligned}
    {\widetilde{\mathcal{E}}}^{n+1} - {\mathcal{E}}^n &= -\frac{1}{\tau}\iaQ{\delta \bQ^{n+1}} - g^n \kappa_1 \|\delta \bQ^{n+1}\|_h^2 - \frac{1}{2}\|\nabla_h \delta \bQ^{n+1}\|_h^2 \\
    &\quad  -\frac{1}{\tau}\iau{\delta \bu^{n+1}} - g^n \kappa_2 \|\delta \bu^{n+1}\|_h^2 - \frac{1}{2}\|\Delta_h \delta \bu^{n+1}\|_h^2 \le 0.
\end{aligned}
\end{equation}
Substituting \eqref{eq34} back into (\ref{eq_diff_E}) and using Lemma~\ref{lem_norm_equivalence}, we arrive at the following key inequality for the provisional modified energy:
\begin{equation}\label{eq_proof1}
\begin{aligned}
    {\widetilde{\mathcal{E}}}^{n+1} - \mathcal{E}^n &= {-\frac{1}{\tau}\iaQ{\delta \bQ^{n+1}} - g^n \kappa_1 \|\delta \bQ^{n+1}\|_h^2 - \frac{1}{2}\|\nabla_h \delta \bQ^{n+1}\|_h^2}\\
    &\quad  {-\frac{1}{\tau}\iau{\delta \bu^{n+1}} - g^n \kappa_2 \|\delta \bu^{n+1}\|_h^2 - \frac{1}{2}\|\Delta_h \delta \bu^{n+1}\|_h^2}\\
    & \leq -\frac{1}{\tau}\ibQ{\delta \bQ^{n+1}}-\frac{1}{\tau}\ibu{\delta \bu^{n+1}}
    = -\mathcal{R}^{n+1} \le 0.
\end{aligned}
\end{equation}

Next, we evaluate the continuous relaxation corrector. By the definition of the modified energy and feasible set $\mathcal{V}_2$, the difference between the final updated energy $\mathcal{E}^{n+1}$ and the provisional energy $\widetilde{\mathcal{E}}^{n+1}$ stems entirely from the modification of the scalar auxiliary variable:
\begin{equation}\label{eq_proof2}
    \mathcal{E}^{n+1} - \widetilde{\mathcal{E}}^{n+1} = s^{n+1} - \tilde{s}^{n+1}\le (1 - \eta_0)\tau \mathcal{R}^{n+1}.
\end{equation}
Adding \eqref{eq_proof1},  and \eqref{eq_proof2} together, the provisional energy terms cancel out perfectly, and we obtain the global energy dissipation law:
\begin{equation}\label{eq_proof4}
    \mathcal{E}^{n+1} - \mathcal{E}^n \le -\mathcal{R}^{n+1} + (1 - \eta_0)\tau\mathcal{R}^{n+1} = -\eta_0 \mathcal{R}^{n+1} \le 0.
\end{equation}
Thus, the modified energy rigorously dissipates at every time step, completely independent of the exact value of the dynamic parameter $\xi^{n+1}$. This completes the proof.
\end{proof}

\begin{lemma}[Uniform Bounds for $g^n$]\label{gn_bound0}
On the basis of the energy stability in Theorem \ref{them3_1},
 there exists a positive constant $G^*$ independent of $\tau$ and $n$ such that the scalar factor $g^n$ defined in \eqref{eq:g_n} satisfies the following uniform bounds for all time steps:
\begin{equation}
    0 \le g^n = \frac{\exp(s^n)}{\exp(E_{1h}^{n+1})} \le G^*< \infty, \quad \forall 0 \le n \le T/\tau.
\end{equation}
\end{lemma}
\begin{proof}
On the basis of
 Theorem \ref{them3_1}, we have for any $n \ge 1$:
\begin{align*}
   s^n \le \mathcal{E}_h(\bQ^n, \bu^n, s^n) &\le \mathcal{E}_h(\bQ^0, \bu^0, s^0) = \frac{K}{2} \|\nabla_h \bQ^0\|_h^2 + B_0 \|\Delta_h \bu^0\|_h^2 + s^0 := C_0.
\end{align*}

 Applying the higher-order estimates established in Theorem \ref{lem3_2} (which provides uniform bounds such as $\|u_h^n\|_{\infty_h} \le C$ and $\|\mathbf{Q}_h^n\|_{\infty_h} \le C$) to the discrete nonlinear energy parts $E_{1h}^{n+1}$, we can apply the discrete Cauchy-Schwarz inequality to deduce that the following bounds:
\begin{align*}
    \left| \langle D_h^2 u_h^n , M_h^n u_h^n \rangle_h \right| &\le \|D_h^2 u_h^n\|_h \|M_h^n\|_{\infty_h} \|u_h^n\|_h \le C, \\
    \| M_h^n u_h^n \|_h^2 &\le \|M_h^n\|_{\infty_h}^2 \|u_h^n\|_h^2 \le C.
\end{align*}
Since the discrete bulk potentials $\langle f_{bn}(\mathbf{Q}_h^n), 1 \rangle_h$ and $\langle f_{s}(u_h^n), 1 \rangle_h$ are bounded from below by their physical definitions, and bounded from above due to the uniform $\infty_h$ bounds on the state variables, we deduce that there exist positive constants $C_*$ and $C^*$ (independent of the spatial step $h$ and time step $\tau$) such that the exact discrete energy $E_{1h}^n$ satisfies:
\begin{equation}\label{eq_E1_bounds}
    -C_* \le E_{1h}^n \le C^*.
\end{equation}
Combining the above bounds for $s^n$ and $E_{1h}^{n+1}$, we have:
\begin{equation}
    0 \le g^n = \frac{\exp(s^n)}{\exp(E_{1h}^{n+1})} \le \frac{\exp(C_0)}{\exp(-C_*)} = G^* < \infty,
\end{equation}which completes the proof.
\end{proof}

\begin{lemma}\label{lem3_2}
    On the basis of the energy stability established in Theorem \ref{them3_1}, there exist positive constants $C$ and $C'$ independent of $\tau$ and $n$ such that the following higher-order estimates hold for the discrete solutions $\bQ^{n}$ and $\bu^{n}$:
 \begin{align}
    \max_{0 \le n \le N-1} \norm{\nabla \bQ^{n+1}}_h^2 + \tau \sum_{k=0}^{n} \left(  \ibQ{\dtau \bQ^{k+1}} + \norm{\Delta \bQ^{k+1}}_h^2 \right) \le C,\\
    \max_{0 \le n \le N-1} \norm{\Delta \bu^{n+1}}_h^2 + \tau \sum_{k=0}^{n} \left(  \ibu{\dtau \bu^{k+1}} + \norm{\Delta^2 \bu^{k+1}}_h^2 \right) \le C',
\end{align}
where $C$ and $C'$ are constants independent of the time step size $\tau$ and the number of time steps $N$.
\end{lemma}
\begin{proof}
Let  $\dtau \bQ^{n+1} = \frac{\bQ^{n+1} - \bQ^n}{\tau}$ and $\dtau \bu^{n+1} = \frac{\bu^{n+1} - \bu^n}{\tau}$.
Taking the inner product of  \eqref{eq_scheme_Q} with $2\dtau \bQ^{n+1}$, we have:
\begin{equation}\label{eq_32}
    2\id{\mathcal{Q}(\tau \mLQ) \dtau \bQ^{n+1}}{\dtau \bQ^{n+1}} + 2\id{\mLQ \bQ^{n+1}}{\dtau \bQ^{n+1}} = 2\id{\mNQ^n}{\dtau \bQ^{n+1}}.
\end{equation}
 Using the standard algebraic identity and the definition of the norm $\iaQ{\cdot}$, we can rewrite the left-hand side of \eqref{eq_32} as follows:
\begin{align}
  2\iaQ{\dtau \bQ^{n+1}}+  2\id{\mLQ \bQ^{n+1}}{\dtau \bQ^{n+1}} =&2\iaQ{\dtau \bQ^{n+1}}+ \frac{1}{\tau} \left( \norm{ \bQ^{n+1}}_{\mLQ}^2 - \norm{ \bQ^n}_{\mLQ}^2 + \norm{\bQ^{n+1} - \bQ^n}_{\mLQ}^2\right) \no  \\=& 2\iaQ{\dtau \bQ^{n+1}}+\tau\norm{\dtau \bQ^{n+1}}_{\mLQ}^2 +\frac{1}{\tau} \left( \norm{ \bQ^{n+1}}_{\mLQ}^2 - \norm{ \bQ^n}_{\mLQ}^2\right)
   \no  \\=&2\ibQ{\dtau \bQ^{n+1}}+\frac{1}{\tau} \left( \norm{ \bQ^{n+1}}_{\mLQ}^2 - \norm{ \bQ^n}_{\mLQ}^2\right).
\end{align}
 Using the Cauchy-Schwarz and Young's inequalities, we can estimate the right-hand side of \eqref{eq_32} as follows:
\begin{equation}
    2\id{\mNQ^n}{\dtau \bQ^{n+1}} \le 2 \norm{\mNQ^n}_h \norm{\dtau \bQ^{n+1}}_h \le C \norm{\mNQ^n}_h^2 + \norm{\dtau \bQ^{n+1}}_h^2.
\end{equation}
Combining the above estimates and using Lemma~\ref{lem_norm_equivalence}, we arrive at the following inequality:
\begin{align} \label{eq_est1}
\frac{1}{\tau} \left( \norm{ \bQ^{n+1}}_{\mLQ}^2 - \norm{ \bQ^n}_{\mLQ}^2\right) + \ibQ{\dtau \bQ^{n+1}}
    \leq  C \norm{\mNQ^n}_h^2 .
\end{align}

We rearrange  \eqref{eq_scheme_Q} to isolate the Laplacian term:
\begin{equation} \label{eq_35}
    \mLQ \bQ^{n+1} = \mNQ^n  - \mathcal{Q}(\tau \mLQ) \dtau \bQ^{n+1}.
\end{equation}
Taking the discrete $l^2$ norm of \eqref{eq_35}, multipling both sides by $\frac{1}{4}$ and using the Lemma~\ref{lem_norm_equivalence}, we have:
\begin{equation} \label{eq_est2}
    \norm{\mLQ \bQ^{n+1}}_h^2 \le 2\norm{\mNQ^n}_h^2 + 2\norm{\mathcal{Q}(\tau \mLQ) \dtau \bQ^{n+1}}_h^2\le 2\norm{\mNQ^n}_h^2 +  2\norm{\dtau \bQ^{n+1}}_{\mathcal{Q}_1}^2.
\end{equation}
Combining estimates~\eqref{eq_est1} and~\eqref{eq_est2}, we have:
\begin{equation}
    \frac{1}{\tau} \left( \norm{ \bQ^{n+1}}_{\mLQ}^2 - \norm{ \bQ^n}_{\mLQ}^2\right) + \frac{1}{2}\ibQ{\dtau \bQ^{n+1}} + \frac{1}{4}\norm{\mLQ \bQ^{n+1}}_h^2
    \leq  C' \norm{\mNQ^n}_h^2.
\end{equation}
Multiplying both sides by $\tau$ and summing over $n$ from $0$ to $N-1$, we obtain:
\begin{equation}
    \norm{ \bQ^{n+1}}_{\mLQ}^2 + \tau \sum_{k=0}^{n} \left(  \frac{1}{2}\ibQ{\dtau \bQ^{k+1}} + \frac{1}{4}\norm{\mLQ \bQ^{k+1}}_h^2 \right) \le C' \tau \sum_{k=0}^{n} \norm{\mNQ^k}_h^2 + \norm{ \bQ^0}_{\mLQ}^2.
\end{equation}

On the basis of the Theorem \ref{them3_1}, we can deduce that there exists a constant $C > 0$ independent of $\tau$ and $n$ such that
\begin{align}
     \norm{\mNQ(\bQ^n,\bu^{n})}_h^2 \leq C. \label{align7_28}
\end{align}
Thus, we conclude that there exists a constant $M > 0$ independent of $\tau$ and $n$ such that the following uniform bound holds for $\bQ^{n}$:
\begin{equation}
    \max_{0 \le n \le N-1} \norm{ \bQ^{n+1}}_{\mLQ}^2 + \tau \sum_{k=0}^{n} \left(  \frac{1}{2}\ibQ{\dtau \bQ^{k+1}} + \frac{1}{4}\norm{\mLQ \bQ^{k+1}}_h^2 \right) \le M.
\end{equation}

Next, we perform a similar analysis for the velocity field $\bu^{n}$.
Similarly to the previous analysis for $\bQ^{n}$, we follow the same steps to derive the  estimate for $\bu^{n}$:
\begin{equation}
    \max_{0 \le n \le N-1} \norm{ \bu^{n+1}}_{\mLu}^2 + \tau \sum_{k=0}^{n} \left(  \frac{1}{2}\ibu{\dtau \bu^{k+1}} + \frac{1}{4}\norm{\mLu \bu^{k+1}}_h^2 \right) \le M',
\end{equation}
where $M'$ is a constant independent of $\tau$ and $n$.
\end{proof}
\begin{remark}
Here, we provide a detailed analysis of the nonlinear term $\mNu(\bQ,\bu)$. To establish the global error estimate, we first bound its discrete spatial $l^2$-norm at each individual time step $t_n$, and subsequently accumulate these bounds over the entire temporal sequence.
    Dropping constant coefficients for brevity, the term scales as $\nabla_h \cdot (\nabla_h \cdot (\bQ u_h))$, which can be expanded  as follows:
\begin{equation}
    \nabla_h \cdot (\nabla_h \cdot (\bQ u_h)) = {(\Delta_h \bQ) u_h}_{} + {2 (\nabla_h \cdot \bQ) \cdot \nabla_h u_h}_{} + {\bQ : D^2_h u_h}_{}. \label{equation7_35}
\end{equation}
Suming the square of the $l^2$-norm of each term in the above expansion over all time steps, we have:
\begin{align*}
    \dt \sum_{n} &\norm{ \nabla_h \cdot (\nabla_h \cdot (\bQ u_h)) }_h^2 \le 3 \dt \sum_{n} \norm{(\Delta_h \bQ^n) \bu^{n}}_h^2 + 3\dt \sum_{n} \norm{\nabla_h \bQ^n \cdot \nabla_h \bu^{n}}_h^2 + 3\dt \sum_{n} \norm{\bQ^n (D^2_h \bu^{n})}_h^2\\&\qquad\qquad
    \le  3 \dt \sum_{n} \norm{\Delta_h \bQ^n}_h^2 \norm{\bu^{n}}_{\infty}^2 + 3\dt \sum_{n} \norm{\nabla_h \bQ^n}_{H_h^1}^2 \norm{\nabla_h \bu^{n}}_{H_h^1}^2 + 3\dt \sum_{n} \norm{\bQ^n}_{\infty}^2 \norm{D^2_h \bu^{n}}_h^2\\&\qquad\qquad
        \le 3  M_u^2 T \left( \dt \sum_{n} \norm{\Delta_h \bQ^n}_h^2 \right) + 3  M_u^2 T \left( \dt \sum_{n} \norm{\bQ^n}_{H_h^2}^2 \right) + 3  M_u^2 T \left( \dt \sum_{n} \norm{\bQ^n}_{H_h^2}^2 \right)\\&\qquad\qquad
        = 3 M_u^2 T \left( \dt \sum_{n} \norm{\Delta_h \bQ^n}_h^2 + 2\dt \sum_{n} \norm{\bQ^n}_{H_h^2}^2\right)  \le C',
\end{align*}
where we have used the uniform bounds for $\bu^n$ and $\bQ^n$ established in Theorem \ref{them3_1} and Lemma~\ref{lem3_2}.
Following a similar analysis for the other nonlinear term, we establish uniform bounds for  $\mNu(\bQ,\bu)$ in the discrete $l^2$-norm over all time steps.
\end{remark}
\begin{lemma}[Uniform Bounds for $g^n$]\label{gn_bound}
On the basis of the energy stability in Theorem \ref{them3_1} and the higher-order estimates established in Lemma~\ref{lem3_2}, we can derive uniform bounds for the scalar factor $g^n$ and the auxiliary variable $s^n$. Specifically,
 there exist positive constants $G_*$ and $G^*$ independent of $\tau$ and $n$ such that
\begin{equation}
    G_* \le g^n = \frac{\exp(s^n)}{\exp(E_{1h}^{n+1})} \le G^*, \quad \forall 0 \le n \le T/\tau.
\end{equation}
\end{lemma}
\begin{proof}
The upper bound for $g^n$ has been established in Lemma~\ref{gn_bound0}, so we focus on deriving a uniform positive lower bound for $g^n$.
From the definition of $s^{n+1}$ in \eqref{eq_scheme_s}, we have:
\begin{equation}
    s^{n+1} - s^n = g^n \left( \langle \muQ^n, \bQ^{n+1} - \bQ^{n} \rangle_h + \langle \mu_h^n, \bu^{n+1} - \bu^{n} \rangle_h \right).
\end{equation}
Applying the Cauchy-Schwarz inequality, we can estimate the right-hand side as follows:
\begin{equation}
    s^{n+1} \ge s^n - |g^n| \tau \left( \norm{\muQ^n}_h \norm{\dtau \bQ^{n+1}}_h + \norm{\mu_h^n}_h \norm{\dtau \bu^{n+1}}_h \right).
\end{equation}
Summing the above inequality from $k=0$ to $n$, we have:
\begin{equation}\label{eq1}
    s^{n+1} \ge s^0 - \sum_{k=0}^{n} \tau |g^k| \norm{\muQ^k}_h \norm{\dtau \bQ^{k+1}}_h - \sum_{k=0}^{n} \tau |g^k| \norm{\mu_h^k}_h \norm{\dtau \bu^{k+1}}_h.
\end{equation}
Let $C_g = \max_k |g^k|$.
Based on the Lemma~\ref{lem3_2}, for \eqref{eq1}, we have:
\begin{align*} \label{eq_lower_bound_step}
    \sum_{k=0}^{n} \tau \norm{\muQ^k}_h \norm{\dtau \bQ^{k+1}}_h
    \le \left( \tau \sum_{k=0}^{n} \norm{\muQ^k}_h^2 \right)^{1/2} \left( \tau \sum_{k=0}^{n} \norm{\dtau \bQ^{k+1}}_h^2 \right)^{1/2}
     \le \sqrt{T} K_{\mu} \sqrt{C} =: C_{drift},\\
    \sum_{k=0}^{n} \tau \norm{\mu_h^k}_h \norm{\dtau \bu^{k+1}}_h
    \le \left( \tau \sum_{k=0}^{n} \norm{\mu_h^k}_h^2 \right)^{1/2} \left( \tau \sum_{k=0}^{n} \norm{\dtau \bu^{k+1}}_h^2 \right)^{1/2}
     \le \sqrt{T} K_{\mu} \sqrt{C'} =: C_{drift},
\end{align*}
 where $K_{\mu}$ is a constant depending on the uniform bounds of $\muQ^k$ and $\mu_h^k$.

Combining the estimates for both $\bQ$ and $\bu$ terms, we have:
\begin{equation}
    s^{n} \ge s^0 - 2 C_g C_{drift} (T), \quad \forall n \ge 0.
\end{equation}
which proves $s^n$ is uniformly bounded from below.
Combining this with the upper bound for $E_1$ in \eqref{eq_E1_bounds}, we conclude that there exists a positive constant $G_*$ such that:
\begin{equation}
    g^n = \frac{\exp(s^n)}{\exp(E_{1h}^{n+1})} \ge \frac{\exp(s^0 - 2 C_g C_{drift} (T))}{\exp(C^*)} =: G_* > 0.
\end{equation}
\end{proof}
\section{Higher-Order Regularity of the EOP-GSAV-EI  Scheme}~\\
In this section, we present a direct regularity bootstrapping estimate for the EOP-GSAV-EI  scheme. The key idea is to leverage the smoothing properties of the linear semigroup generated by $\mLQ$ and $\mLu$ to obtain higher-order regularity estimates for  $\bQ^n$ and $\bu^n$.
\begin{lemma} \label{lem_analytic_semigroup}
Under periodic boundary conditions, both operators generate discrete analytic semigroups $e^{-tA}$ satisfying the smoothing estimate:
\begin{equation} \label{eq_smoothing_estimate}
    \vertiii{A^\beta e^{-tA}}_h \le C_\beta t^{-\beta}, \quad \forall t > 0, \; \beta > 0,
\end{equation}
where $\vertiii{\cdot}_h$ denotes the induced operator norm corresponding to the discrete $l^2$-norm, $A \in \{\mLQ, \mLu\}$, and the constant $C_\beta > 0$ depends only on $\beta$ and is independent of the mesh size $h$.
\end{lemma}

\begin{proof}
By classical semigroup theory \cite[Chap.~2, Sec.~6]{pazy2012semigroups},, if a finite-dimensional operator $A$ is self-adjoint and positive semi-definite, functional calculus yields $\norm{A^\beta e^{-tA}}_h \le \sup_{\lambda \ge 0} \lambda^\beta e^{-t\lambda} = (\beta/e)^\beta t^{-\beta}$. Thus, the estimate \eqref{eq_smoothing_estimate} holds with $C_\beta = (\beta/e)^\beta$, which is inherently independent of $h$. It remains to verify that $\mLQ$ and $\mLu$ are self-adjoint and positive semi-definite under the discrete inner product $\langle \cdot, \cdot \rangle_h$.

For $\mLQ$, applying discrete summation by parts (SBP) yields:
\begin{equation*}
    \langle \mLQ u_h, v_h \rangle_h = \langle \nabla_h u_h, \nabla_h v_h \rangle_h+g^n \kappa_1 \langle u_h,  v_h \rangle_h = \langle u_h, \mLQ v_h \rangle_h,
\end{equation*}
which implies $\mLQ$ is self-adjoint. Setting $v_h = u_h$ gives $\langle \mLQ u_h, u_h \rangle_h = \norm{\nabla_h u_h}_h^2 \ge 0$, confirming it is positive semi-definite.

Similarly, applying SBP twice for $\mLu$ yields:
\begin{equation*}
    \langle \mLu u_h, v_h \rangle_h = \langle \Delta_h u_h, \Delta_h v_h \rangle_h+g^n \kappa_2 \langle u_h,  v_h \rangle_h = \langle u_h, \mLu v_h \rangle_h,
\end{equation*}
and $\langle \mLu u_h, u_h \rangle_h = \norm{\Delta_h u_h}_h^2 \ge 0$. Thus, $\mLu$ also satisfies the required properties, completing the proof.
\end{proof}
\begin{theorem} \label{th4_1}
    Aussuming the initial data $\bQ^0$ and $\bu^0$ possess sufficient regularity  $\bQ^0 \in H^{2}(\Omega)$ and $\bu^0 \in H^{4}(\Omega)$.
Under the assumptions of Lemma~3.3, for   $0<\epsilon < \frac{1}{4}$, there exists a constant $C_\epsilon$ independent of $n$ such that the following regularity estimate holds for the numerical solutions generated by the EOP-GSAV-EI  scheme:
\begin{equation}
    \sup_{0 \le n \le N} \|\bQ^n\|_{H_h^{2}}+  \sup_{0 \le n \le N} \norm{\Delta_h \bu^n}_{h} \le C_\epsilon, \quad \forall n \ge 0.
\end{equation}
And the following uniform bound holds for the nonlinear term $\mNQ^n$ and $\mNu^n$:
\begin{equation}
    \sup_{0 \le n \le N} \norm{\mNQ^n}_h + \sup_{0 \le n \le N} \norm{\mNu^n}_h \le C.
\end{equation}
\end{theorem}
\begin{proof}
 To rigorously derive the global error representation, we first recall the exact single-step evolution of the semi-discrete scheme over the interval $[t_k, t_{k+1}]$. By applying the variation-of-constants formula and explicitly treating the nonlinear term as a constant $\mNQ^k = \mNQ(\bQ^k)$ within this local step, we have:
\begin{equation} \label{eq_single_step}
    \bQ^{k+1} = e^{\tau \mLQ} \bQ^k + \int_0^\tau e^{(\tau-s)\mLQ} \mNQ^k \, ds,
\end{equation}
where $\tau = t_{k+1} - t_k$ is the uniform time step size.

Iteratively applying \eqref{eq_single_step} from $k=0$ up to $n-1$, we unroll the recurrence relation to obtain the fully discrete Duhamel's formula. This expresses $\bQ^n$ as the sum of the initial state's zero-input evolution and the accumulated historical nonlinear responses:
\begin{equation} \label{eq_discrete_sum}
    \bQ^n = e^{t_n \mLQ} \bQ^0 + \sum_{k=0}^{n-1} e^{(t_n - t_{k+1}) \mLQ} \int_0^\tau e^{(\tau-s)\mLQ} \mNQ^k \, ds.
\end{equation}

To facilitate the subsequent analytic semigroup estimates, it is highly advantageous to rewrite this discrete summation as a single continuous integral over the entire time domain $[0, t_n]$. We introduce a piecewise constant interpolation of the nonlinear term in time, defined as:
\begin{equation} \label{eq_piecewise_N}
    \widetilde{\mNQ}(\sigma) := \mNQ^k, \quad \text{for } \sigma \in [t_k, t_{k+1}).
\end{equation}

By utilizing the change of variables $\sigma = s + t_k$ within each local integral, we observe that the differential becomes $d\sigma = ds$, and the local integration bounds $[0, \tau]$ map exactly to the global sub-interval $[t_k, t_{k+1}]$. Furthermore, the temporal indices in the exponential operators can be perfectly merged. Recalling that $\tau = t_{k+1} - t_k$, the exponent simplifies algebraically as follows:
\begin{equation*}
    (t_n - t_{k+1}) + (\tau-s) = t_n - t_{k+1} + (t_{k+1} - t_k) - (\sigma - t_k) = t_n - \sigma.
\end{equation*}

Consequently, the summation of local integrals identically fuses into a global continuous integral representation:
\begin{equation} \label{eq_discrete_duhamel}
    \bQ^n = \underbrace{e^{t_n \mLQ} \bQ^0}_{I_1} + \underbrace{\int_0^{t_n} e^{(t_n - \sigma)\mLQ} \widetilde{\mNQ}(\sigma) \, d\sigma}_{I_2}.
\end{equation}

Let $\mathcal{A} = -\mLQ$ be a strictly positive sectorial operator.
 Applying the operator $\mathcal{A}^{1 - \frac{\epsilon}{2}}$  to both sides of \eqref{eq_discrete_duhamel} and taking the $l^2$-norm, we can obtain the following estimate:
\begin{equation}\label{eq_triangle_I1_I2}
    \norm{\mathcal{A}^{1 - \frac{\epsilon}{2}} \bQ^n}_h \le \norm{\mathcal{A}^{1 - \frac{\epsilon}{2}} I_1}_h + \norm{\mathcal{A}^{1 - \frac{\epsilon}{2}} I_2}_h.
\end{equation}
For the initial value term $I_1$, on the basis of the assumption $\bQ^0 \in H^{2}(\Omega)$, we have
\begin{equation}
    \norm{\mathcal{A}^{1 - \frac{\epsilon}{2}} I_1}_h = \norm{e^{t_n \mLQ} \mathcal{A}^{1 - \frac{\epsilon}{2}} \bQ^0}_h \le C \norm{\mathcal{A}^{1 - \frac{\epsilon}{2}} \bQ^0}_h \le C_0.
\end{equation}
For the integral term $I_2$, we use the Lemma~\ref{lem_analytic_semigroup} and \eqref{align7_28} to the integrand:
\begin{align}
    \norm{\mathcal{A}^{1 - \frac{\epsilon}{2}} I_2}_h &\le \int_0^{t_n} \norm{\mathcal{A}^{1-\frac{\epsilon}{2}} e^{(t_n - \sigma)\mLQ} \widetilde{\mNQ}(\sigma)}_h \, d\sigma \notag \\
    &\le C_{1-\frac{\epsilon}{2}} \int_0^{t_n} (t_n - \sigma)^{-(1-\frac{\epsilon}{2})} \norm{\widetilde{\mNQ}(\sigma)}_h \, d\sigma \notag \\
    &\le C_{1-\frac{\epsilon}{2}} M_{\mNQ} \int_0^{t_n} (t_n - \sigma)^{-1+\frac{\epsilon}{2}} \, d\sigma\no
    \le \frac{2 C_{1-\frac{\epsilon}{2}} M_{\mNQ}}{\epsilon} T^{\frac{\epsilon}{2}}.
\end{align}
Combining the estimates for both $I_1$ and $I_2$, we conclude that there exists a constant $C_\epsilon$ independent of $n$ such that:
\begin{equation}\label{eq3_8}
    \sup_{0 \le n \le N} \norm{\ma^{1 - \frac{\epsilon}{2}} \bQ^n}_{h} \le C\left(\norm{\ma^{1 - \frac{\epsilon}{2}} \bQ^0}_{h}, T, M_{\mNQ}, \epsilon \right).
\end{equation}

Next, we perform a similar analysis for the velocity field $\bu^n$.
Iterating the equation for $\bu^n$ in a similar manner, we can derive the following discrete Duhamel's formula for $\bu^n$:
\begin{equation}
    \bu^n = e^{t_n \mLu} \bu^0 + \int_0^{t_n} e^{(t_n - s)\mLu} \widetilde{\mNu}(s) \, ds, \label{eq_v}
\end{equation}
where $\widetilde{\mNu}(s) = \mNu^k$ for $s \in [t_k, t_{k+1})$.

To rigorously bound the highly singular double divergence term within the nonlinear function, we employ a discrete duality argument combined with the fractional powers of the operator $\ma$. By the standard variational definition of the discrete $L^2$-norm and the self-adjointness of $\ma^{-\frac{\ep}{2}}$, we have:
\begin{equation} \label{eq_dual_def_divdiv}
    \begin{aligned}
        \norm{\ma^{-\frac{\ep}{2}} \nabla_h \cdot (\nabla_h \cdot (\bM_h \bu))}_h
        &= \sup_{v_h \ne 0} \frac{\langle \ma^{-\frac{\ep}{2}} \nabla_h \cdot (\nabla_h \cdot (\bM_h \bu)), v_h \rangle_h}{\norm{v_h}_h} \\
        &= \sup_{v_h \ne 0} \frac{\langle \nabla_h \cdot (\nabla_h \cdot (\bM_h \bu)), \ma^{-\frac{\ep}{2}} v_h \rangle_h}{\norm{v_h}_h}.
    \end{aligned}
\end{equation}

Under periodic boundary conditions, performing discrete summation by parts (SBP) twice allows us to shift the spatial derivatives onto the test function without generating any boundary terms. Crucially, since $\ma$ is a constant-coefficient operator generated by the discrete Laplacian, it commutes with the discrete differential operators. This property allows us to freely redistribute the fractional operators before applying the Cauchy-Schwarz inequality:
\begin{align} \label{eq_sbp_cauchy}
    \langle \nabla_h \cdot (\nabla_h \cdot (\bM_h \bu)), \ma^{-\frac{\ep}{2}} v_h \rangle_h
    &= \langle \bM_h \bu, \nabla_h \nabla_h (\ma^{-\frac{\ep}{2}} v_h) \rangle_h \nonumber \\
    &= \langle \ma^{1-\frac{\ep}{2}} (\bM_h \bu), \ma^{-1} \nabla_h \nabla_h v_h \rangle_h \nonumber \\
    &\le \norm{\ma^{1-\frac{\ep}{2}} (\bM_h \bu)}_h \norm{\ma^{-1} \nabla_h \nabla_h v_h}_h.
\end{align}

Finally, we invoke discrete elliptic regularity, which is equivalent to the uniform $L^2$-boundedness of the discrete Riesz transforms ($\ma^{-1} \nabla_h \nabla_h$). By Parseval's identity in the discrete Fourier domain, this guarantees the existence of a positive constant $C$, independent of the mesh size $h$, such that $\norm{\ma^{-1} \nabla_h \nabla_h v_h}_h \le C \norm{v_h}_h$. Substituting this uniform bound back into \eqref{eq_sbp_cauchy} and taking the supremum over $v_h$ in \eqref{eq_dual_def_divdiv}, we obtain the crucial estimate:
\begin{equation} \label{eq_double_div_bound}
    \norm{\ma^{-\frac{\ep}{2}} \nabla_h \cdot (\nabla_h \cdot (\bM_h \bu))}_h \le C \norm{\ma^{1-\frac{\ep}{2}} (\bM_h \bu)}_h.
\end{equation}
The other terms in $\mNu(\bQ^n, \bu^n)$ can be also estimated, leading to the following bound:
\begin{equation} \label{eq_Nu_bound}
    \norm{\ma^{-\frac{\ep}{2}} \mNu(\bQ^n, \bu^n)}_{h}  \le C \norm{\ma^{-\frac{\ep}{2}} \bQ^n}_{h} \norm{\bu^n}_{H_h^{2}}.
\end{equation}

Let $\mathcal{B} = -\mLu$ be a strictly positive sectorial operator. Applying the operator $\ma^{-\frac{\ep}{2}}\mathcal{B}^{1-\frac{\epsilon}{4}}$ to both sides of \eqref{eq_v}, taking the $l^2$-norm, and simultaneously utilizing the Lemma~\ref{lem_analytic_semigroup} along with the nonlinear bound in \eqref{eq_Nu_bound}, we obtain the following combined estimate:
\begin{equation} \label{eq_combined_estimate}
    \begin{aligned}
        \norm{\ma^{-\frac{\ep}{2}}\mathcal{B}^{1-\frac{\epsilon}{4}} \bu^n}_{h}
        &\le \norm{e^{t_n \mathcal{B}} \ma^{-\frac{\ep}{2}}\mathcal{B}^{1-\frac{\epsilon}{4}} \bu^0}_{h} + \norm{\int_0^{t_n} \ma^{-\frac{\ep}{2}}\mathcal{B}^{1-\frac{\epsilon}{4}} e^{(t_n - s)\mathcal{B}}  \widetilde{\mNu}(s) \, ds}_{h} \\
        &\le C e^{-\lambda t_n} \norm{\ma^{-\frac{\ep}{2}}\mathcal{B}^{1-\frac{\epsilon}{4}} \bu^0}_{h} + C \int_0^{t_n} (t_n-s)^{-\left(1 - \frac{\epsilon}{4}\right)} \norm{\widetilde{\mNu}(s)}_{h} \, ds \\
        &\le C \norm{\ma^{-\frac{\ep}{2}}\mathcal{B}^{1-\frac{\epsilon}{4}} \bu^0}_h + C \sup_{s \in [0,t_n]} \norm{\ma^{-\frac{\ep}{2}}\widetilde{\mNu}(s)}_{h} \int_0^{t_n} (t_n-s)^{-1 + \frac{\epsilon}{4}} \, ds \\
        &\le C_0 + C \sup_{s \in [0,t_n]} \norm{\ma^{-\frac{\ep}{2}}\widetilde{\mNu}(s)}_{h} \cdot \frac{4}{\epsilon} t_n^{\frac{\epsilon}{4}}.
    \end{aligned}
\end{equation}
Since $t_n \le T$, the right-hand side is uniformly bounded. Thus, recalling the norm equivalence associated with $\mathcal{B}$, there exists a constant $C$ independent of $n$ such that:
\begin{equation}\label{eq4_13}
    \sup_{0 \le n \le N} \norm{\ma^{-\frac{\ep}{2}}\mathcal{B}^{1-\frac{\epsilon}{4}}\bu^n}_{h} \le C\left(\norm{\ma^{-\frac{\ep}{2}}\mathcal{B}^{1-\frac{\epsilon}{4}}\bu^0}_{h}, T, M_{\mNu}, \epsilon \right).
\end{equation}
Thus, we have established a uniform bound for $\Delta_h\bu^n$ in the fractional Sobolev space $l^\infty$:
\begin{equation}
    \sup_{0 \le n \le N} \norm{\Delta_h \bu^n}_{h} \le C.
\end{equation}
Consequently, by \eqref{eq4_13}, we immediately deduce $\sup_{0 \le n \le N} \norm{\ma^{\frac{\ep}{2}}\mNQ^n}_{h} \le C$.
Applying the smoothing estimate \eqref{eq_triangle_I1_I2} again and applying the $\ma$, we can further lift the regularity of $\bQ^n$ to $H_h^2$:
\begin{equation}
  \norm{\bQ^n}_{H_h^{2}} \le C \norm{A^{1 - \frac{\epsilon}{2}}\ma^{\frac{\ep}{2}}\bQ^n}_{h} \le C\left(\norm{\bQ^0}_{H_h^{2}}, T, M_{\mNQ}, \epsilon \right).
\end{equation}
Finally, we can establish the desired regularity estimate for both $\mNQ^n$ and $\mNu^n$:
\begin{equation}
    \sup_{0 \le n \le N} \norm{\mNQ^n}_h + \sup_{0 \le n \le N} \norm{\mNu^n}_h \le C.
\end{equation}
\end{proof}
\begin{lemma}\label{lem:nonlinear_bound}
    Let $\mathbf{Q}_h$ be a discrete approximation of the system \eqref{eq_scheme_Q}-\eqref{eq_scheme_s} in spatial dimension $d \in \{2, 3\}$. Let the dimensional coefficients be $b_2 = 0$ and $b_3 = \frac{|B|}{\sqrt{6}}$.

    If the stabilization parameter $\kappa$ is chosen to satisfy the lower bound condition
$\kappa \ge \norm{\tilde{A}}_\infty+C \norm{\bQ}_\infty  ,$
    then the stabilized nonlinear term $\mathbf{H}_h$ in the Smectic model obeys the following pointwise Frobenius norm estimate:
    \begin{equation}\label{eq_lemma_F_bound}
        \big| \kappa \mathbf{Q}_h + \mathbf{H}_h \big|_F \leq \kappa\big| \mathbf{Q}_h \big|_F + f_d(\big| \mathbf{Q}_h \big|_F),
    \end{equation}
    where the scalar function is defined as $f_d(\xi) := -A\xi + b_d \xi^2 - C\xi^3 + S$.
\end{lemma}
\begin{proof}
Based on the Smectic model, we define the nonlinear operator $\mathbf{H}_h$ element-wise:
\begin{equation*}
\mathbf{H}_h^{ij} = -A Q_h^{ij} + B \left( (\mathbf{Q}_h^2)^{ij} - \frac{1}{d}\text{tr}(\mathbf{Q}_h^2)\delta^{ij} \right) - C \text{tr}(\mathbf{Q}_h^2)Q_h^{ij} - T_1^{ij} - T_2^{ij},
\end{equation*}
where $\mathbf{T}_1 = \frac{2B_0 q^2}{s_+} \left( u_h \cdot D^2 u_h - \frac{\text{tr}(u_h \cdot D^2 u_h)}{d}\mathbf{I} \right)$ and $\mathbf{T}_2 = \frac{2B_0 q^4}{s_+^2} u_h^2 \mathbf{Q}_h$.

On the basis of the Theorem \ref{th4_1}, we define the spatially varying coefficient and the upper bound for the forcing term as:
\begin{align*}
\tilde{A} = A + \frac{2B_0 q^4}{s_+^2} u_h^2,\qquad
S = \left\| \frac{2B_0 q^2}{s_+} ( u_h \cdot D^2 u_h - \frac{\text{tr}(u_h \cdot D^2 u_h)}{d}\mathbf{I})  \right\|_{\infty}.
\end{align*}

To present a unified proof for both two- and three-dimensional cases ($d=2, 3$), we first exploit the fundamental algebraic properties of $d \times d$ traceless symmetric matrices. Specifically, the traceless part of the quadratic term vanishes identically in 2D since $\mathbf{Q}_h^2 = \frac{1}{2}\text{tr}(\mathbf{Q}_h^2)\mathbf{I}$, whereas in 3D, it satisfies the identity $\big| \mathbf{Q}_h^2 - \frac{1}{3}\text{tr}(\mathbf{Q}_h^2)\mathbf{I} \big|_F = \frac{1}{\sqrt{6}} |\mathbf{Q}_h|_F^2$.

By introducing a dimension-dependent parameter $b_d$ (where $b_2 = 0$ and $b_3 = \frac{|B|}{\sqrt{6}}$) and under the condition $\kappa \ge \norm{\tilde{A}+C\bQ}_\infty$, we first apply the triangle inequality to the nonlinear term:
\begin{align*}
    \big| \kappa \mathbf{Q}_h + \mathbf{H}_h \big|_F &\leq \big| (\kappa - \tilde{A} - C|\mathbf{Q}_h|_F^2)\mathbf{Q}_h \big|_F + b_d |\mathbf{Q}_h|_F^2 + |\mathbf{T}_1|_F\\
    &\leq (\kappa - {A} - C|\mathbf{Q}_h|_F^2)|\mathbf{Q}_h|_F + b_d |\mathbf{Q}_h|_F^2 + |\mathbf{T}_1|_F \\
    &\leq \kappa |\mathbf{Q}_h|_F + {f}_d(|\mathbf{Q}_h|_F),
\end{align*}
where $ {f}_d(\xi) := -A\xi + b_d \xi^2 - C\xi^3 + S.$
\end{proof}
\begin{lemma}\label{lemma4_4}
    Consequently, by combining the properties of the scalar function $f_d$ with the \textit{a priori} assumption $\kappa \ge  \max_{\xi \in [0, \eta^{(d)}]} (A  - 2b_d\xi+3C\xi^2)$, we rigorously obtain the desired uniform bound:
    \begin{equation}\label{eq_lemma_final_bound}
        | \kappa \xi + f_d(\xi) | \leq \kappa \eta^{(d)}, \quad \forall \xi \in [0, \eta^{(d)}].
    \end{equation}
\end{lemma}
\begin{proof}
    Analyzing the structure of $f_d(\xi) = -A\xi + b_d \xi^2 - C\xi^3 + S$, we observe that its leading-order term is $-C\xi^3$ (with $C>0$). Therefore, there inherently exists a sufficiently large positive constant $\eta^{(d)}$-depending solely on the physical parameters and the source term bound $S$-such that $f_d(\eta^{(d)}) \le 0$.

Let $g(\xi) := \kappa \xi + f_d(\xi)$ on the interval $[0, \eta^{(d)}]$. To establish the uniform bound $\kappa \eta^{(d)}$, the auxiliary function $g(\xi)$ must be monotonically non-decreasing on $[0, \eta^{(d)}]$.  Requiring $g'(\xi) \ge 0$ yields:
\begin{equation*}
    g'(\xi) = \kappa - A + 2b_d\xi - 3C\xi^2 \ge 0, \quad \forall \xi \in [0, \eta^{(d)}],
\end{equation*}
which imposes the second condition on the stabilization parameter $\kappa$:
\begin{equation}\label{eq_kappa_cond2}
    \kappa \ge  \max_{\xi \in [0, \eta^{(d)}]} (A - 2b_d\xi + 3C\xi^2).
\end{equation}
    Provided that \eqref{eq_kappa_cond2} holds, $g(\xi)$ reaches its maximum at the right endpoint, implying $g(\xi) \le g(\eta^{(d)}) = \kappa \eta^{(d)} + f_d(\eta^{(d)}) \le \kappa \eta^{(d)}$ for all $\xi \in [0, \eta^{(d)}]$. This completes the proof of the uniform bound.
\end{proof}
\begin{theorem}\label{thm:uniform_bound}
    Let $\mathbf{Q}_h^n$ be the numerical solution generated by the ETD1 scheme for the Smectic model. Assume that the initial data satisfies the bounded region condition $\|\mathbf{Q}_h^0\|_{\infty} \leq \eta^{(d)}$, where the truncation parameter $\eta^{(d)} > 0$ is chosen to be sufficiently large (as required by Lemma~\ref{lemma4_4}). If the stabilization parameter $\kappa_1 \geq \frac{1}{G_*} \kappa_0$,
     then the numerical solution unconditionally preserves this maximum norm bound for all time steps, i.e.,
    \begin{equation}
        \|\mathbf{Q}_h^n\|_{\infty} \leq \eta^{(d)}, \quad \forall n \geq 0,
    \end{equation}
\end{theorem}
where
    \begin{equation}\label{eq_final_kappa}
        \kappa_0= \max \left\{ \sup_x \tilde{A}(x) + C(\eta^{(d)})^2, \,\,   \max_{\xi \in [0, \eta^{(d)}]} (A  - 2b_d\xi+3C\xi^2) \right\}.
    \end{equation}
\begin{proof}
We proceed by mathematical induction. Recalling the definitions from \eqref{eq_L} and \eqref{eq_N_Q_def}, the discrete linear operator and the modified nonlinear source term are respectively given by:
\begin{equation*}
    \mathcal{L}_h = \Delta_h + g^n\kappa_1 \mathcal{I}_h, \qquad \mathbf{N}_h^n = g^n \kappa_1 \mathbf{Q}_h^n + \mathbf{H}_h^n.
\end{equation*}

From \cite{du2021}, under periodic boundary conditions, $\Delta_h$ satisfies the discrete maximum principle, which guarantees that it generates a contraction semigroup in the discrete $L^\infty$-norm. Specifically, there exists a positive constant $\lambda$ such that for any $t \ge 0$ and any grid function $v_h$, we have:
\begin{equation}\label{eq_linear_propagator}
    \| e^{t \Delta_h} v_h \|_{\infty} \le e^{-\lambda t} \|v_h\|_{\infty}.
\end{equation}
By Lemmas \ref{gn_bound} and \ref{lemma4_4}, the dynamic variable $g^n$ has a uniform positive lower bound $g^n \ge G_* > 0$ for all $n \ge 0$. Thus, by choosing the stabilization constant $\kappa_1 \ge \kappa_0 / G_*$, we naturally guarantee the effective stabilization condition $g^n \kappa_1 \ge \kappa_0$, where $\kappa_0$ is defined in \eqref{eq_final_kappa}.

Under the \textit{a priori} assumption $\|\mathbf{Q}_h^n\|_{\infty} \leq \eta^{(d)}$, applying Lemma~\ref{lem:nonlinear_bound} directly yields the following uniform bound for the modified nonlinear term:
\begin{equation}\label{eq_nonlinear_uniform}
    \|\mNQ^n\|_{\infty}\leq \norm{g^n \kappa_1 \big| \mathbf{Q}_h^n \big|_F + f_d(\big| \mathbf{Q}_h^n \big|_F)}_\infty \leq g^n \kappa_1 \eta^{(d)} +  f_d(\eta^{(d)}) \leq g^n \kappa_1 \eta^{(d)}.
\end{equation}

    Now, utilizing the integral representation of \eqref{eq_scheme_Q}, we have:
    \begin{equation*}
        \mathbf{Q}_h^{n+1} = e^{-\tau \mathcal{L}_h} \mathbf{Q}_h^{n} + \int_0^\tau e^{-(\tau-s)\mathcal{L}_h} \mNQ^n \, ds.
    \end{equation*}
    Taking the discrete $L^\infty$-norm on both sides, together with the bounds \eqref{eq_linear_propagator} and \eqref{eq_nonlinear_uniform}, we deduce:
    \begin{equation*}
        \begin{aligned}
            \|\mathbf{Q}_h^{n+1}\|_{\infty}
            &\leq \| e^{-\tau \mathcal{L}_h} \mathbf{Q}_h^{n} \|_{\infty} + \int_0^\tau \big\| e^{-(\tau-s)\mathcal{L}_h} \mNQ^n \big\|_{\infty} ds \\
            &\leq e^{-g^n\kappa_1 \tau} \|\mathbf{Q}_h^{n}\|_{\infty} + \int_0^\tau e^{-g^n\kappa_1(\tau-s)} \|\mNQ^n\|_{\infty} \, ds \\
            &\leq e^{-g^n\kappa_1 \tau} \eta^{(d)} + \int_0^\tau e^{-g^n\kappa_1(\tau-s)} \left(g^n\kappa_1 \eta^{(d)}\right) ds.
        \end{aligned}
    \end{equation*}
    Evaluating the straightforward definite integral yields:
    \begin{equation*}
        \int_0^\tau e^{-g^n\kappa_1(\tau-s)} \left(g^n\kappa_1 \eta^{(d)}\right) ds = g^n\kappa_1 \eta^{(d)} \left[ \frac{1}{g^n\kappa_1} e^{-g^n\kappa_1(\tau-s)} \right]_{s=0}^{s=\tau} = \eta^{(d)} \left( 1 - e^{-g^n\kappa_1 \tau} \right).
    \end{equation*}
    Substituting this back into the inequality, we perfectly close the induction:
    \begin{equation*}
        \|\mathbf{Q}_h^{n+1}\|_{\infty} \leq e^{-g^n\kappa_1 \tau} \eta^{(d)} + \eta^{(d)} - e^{-g^n\kappa_1 \tau} \eta^{(d)} = \eta^{(d)}.
    \end{equation*}
    This completes the proof.
\end{proof}
\section{Convergence Analysis of the Fully Discrete Scheme}
~\\
First, we define the error $\boldsymbol{E}_{\mathbf{Q}}^{n},e_{u}^{n},e_{s}^{n}$ as the difference between the numerical solution  at time $t_n$:
\begin{align*}
\boldsymbol{E}_{\mathbf{Q}}^{n} &= \mathbf{Q}_h^n -  \mathbf{Q}(t_n),\quad
e_{u}^{n} = u_h^n -  u(t_n),\quad
e_{s}^{n} = s^n -  s(t_n),\quad
\tilde{e}_{s}^{n} = \tilde{s}_h^n -  s(t_n),\\
\ea &= \frac{\boldsymbol{E}_{\mathbf{Q}}^{n+1} - \boldsymbol{E}_{\mathbf{Q}}^{n}}{\tau},\quad
\eb = \frac{e_{u}^{n+1} - e_{u}^{n}}{\tau},\quad
\ec = \frac{e_{s}^{n+1} - e_{s}^{n}}{\tau},\\
\delta_t \mathbf{Q}(t_{n+1}) &= \frac{\mathbf{Q}(t_{n+1}) - \mathbf{Q}(t_n)}{\tau},\quad
\delta_t u(t_{n+1}) = \frac{u(t_{n+1}) - u(t_n)}{\tau}.
\end{align*}
By subtracting the exact equations from the fully discrete scheme
and applying the transformation $\mq(\tau \mLQ)$ to both sides of the $\mathbf{Q}$-error equation and $\mq(\tau \mLu)$ to the $u$-error equation, we obtain the following transformed error equations:
\begin{align}
\mq(\tau \mLQ) &\ea + \mLQ \boldsymbol{E}_{\mathbf{Q}}^{n+1} = \mNQ^n-\mNQz(t_n)  - \mathbf{T}_{\mathbf{Q}}^n,\label{eq_error_s_final}\\
    \mq(\tau \mLu) & \eb + \mLu e_{u}^{n+1} =\mNu^n-\mNuz(t_n)  - T_u^{n},\label{eq_error_s_final2}\\
    \frac{1}{\tau} (\tilde{e}_{s}^{n+1}-e_{s}^{n}) &= g^n \left( \left\langle \muQ^n, \ea \right\rangle_h + \left\langle \mu_{\mathbf{u}}^n, \eb \right\rangle_h \right) \no\\
    &\quad + \left\langle g^n \muQ^n-\muQz(t_n), \da  \right\rangle_h + \left\langle g^n \muu^n-\muuz(t_n), \db \right\rangle_h - T_s^n,\label{eq_error_s_final3}\\
    {e}_s^{n+1} &
    = \xi^{n+1} {\left( \tilde{e}_s^{n+1} \right)} + (1 - \xi^{n+1}) {\left( E_{1h}^{n+1} - s(t_{n+1}) \right)}.\label{eq_error_s_final4}
\end{align}
Here and in what follows, the analogous notation convention is applied to other operators and variables.
The corresponding truncation errors $\mathbf{T}_{\mathbf{Q}}^n$, $T_u^n$, and $T_s^n$, arising from the time derivative and Laplace operator discretizations of $\mathbf{Q}$ and $u$, and the time derivative discretizations of   $s$, respectively, can be estimated as follows:
\begin{equation} \label{eq_truncation_estimate}
    \|\mathbf{T}_{\mathbf{Q}}^n\|_h + \|T_u^n\|_h  \le C (\tau + h^2), \quad |T_s^n| \le C \tau.
\end{equation}

\begin{lemma}[Lipschitz Continuity of Nonlinear Terms] \label{lem:Lipschitz}
Under the assumption that the exact solution $(\mathbf{Q}(t), u(t))$ is sufficiently smooth, it follows from Theorem \ref{th4_1} that the discrete nonlinear operators $\muQ$ and $\muu$ satisfy the following Lipschitz continuity conditions:
\begin{align}
        \| \muQ(\mathbf{Q}_h^n, u_h^n) - \muQz(\mathbf{Q}(t_n), u(t_n)) \|_h &\le C_{L_1} \left( \|\boldsymbol{E}_{\mathbf{Q}}^{n}\|_h + \|e_{u}^{n}\|_{H^2_h}+ h^2 \right),\label{eq_Lip_Q}\\
    \| \muu(\mathbf{Q}_h^n, u_h^n) - \muuz(\mathbf{Q}(t_n), u(t_n)) \|_h &\le C_{L_2} \left( \|\boldsymbol{E}_{\mathbf{Q}}^{n}\|_{H^2_h} + \|e_{u}^{n}\|_{H^2_h} + h^2 \right),\label{eq_Lip_Q2}
\end{align}
where the positive constants $C_{L_1}, C_{L_2}$ depend only on $M$, $\Omega$, and the exact solution.
\end{lemma}
\begin{proof}
We first analyze the Lipschitz continuity for $\muQ$. The nonlinear term $\mNQ$  can be decomposed into the polynomial part and the derivative coupling part as follows:
\begin{equation}
    \begin{aligned}
         \muQ(\mathbf{Q}_h^n, u_h^n) =
        & \underbrace{ A \mathbf{Q}_h^n + B \left( (\mathbf{Q}_h^n)^2 - \frac{\mathrm{tr}((\mathbf{Q}_h^n)^2)}{3}\mathbf{I} \right) - C \mathrm{tr}((\mathbf{Q}_h^n)^2)\mathbf{Q}_h^n +\left[ - \frac{2B_0 q^4}{s_+^2} \mathbf{Q}_h^n (u_h^n)^2 \right] }_{\muQa(\mathbf{Q}_h^n, u_h^n)}\\
        & + \underbrace{ \left[ - \frac{2B_0 q^2}{s_+} \left( u_h^n D_h^2 u_h^n - \frac{\mathrm{tr}(u_h^n D_h^2 u_h^n)}{3}\mathbf{I} \right) \right] }_{\muQb(u_h^n)}.
    \end{aligned}
\end{equation}
For the polynomial part $\muQa$, taking the discrete $l^2$-norm of the difference along with Theorem \ref{th4_1} yields the following bound:
\begin{align}
    \| \muQa^n-\muQaz(t_n)\|_h &\le A\|\boldsymbol{E}_{\mathbf{Q}}^{n}\|_h + B \left( \|\mathbf{Q}_h^n\|_{\infty} + \|\mathbf{Q}(t_n)\|_{\infty} \right) \|\boldsymbol{E}_{\mathbf{Q}}^{n}\|_h + C \left( \|\mathbf{Q}_h^n\|_{\infty}^2 + \|\mathbf{Q}(t_n)\|_{\infty}^2 \right) \|\boldsymbol{E}_{\mathbf{Q}}^{n}\|_h \no \\
    &\quad + C q^4 s_+^{-2} M^2 \|e_{u}^{n}\|_h + C q^4 s_+^{-2} M \|\boldsymbol{E}_{\mathbf{Q}}^{n}\|_h,\no \\
    &\le C(M) \left( \|\boldsymbol{E}_{\mathbf{Q}}^{n}\|_h + \|e_{u}^{n}\|_h \right).
\end{align}

For the derivative coupling part $\muQb^n$, a direct estimation between the discrete and continuous terms mixes spatial and numerical errors. By rigorously introducing the grid restriction operator $\mathcal{I}_h$, we first apply the triangle inequality to separate the spatial truncation error from the purely discrete algebraic difference:
\begin{align} \label{eq_muQb_triangle}
    \| \muQb^n - \muQbz(t_n) \|_h
    &\le \underbrace{ \left\| \muQbz(t_n) - \muQb^n(\mathcal{I}_h u(t_n)) \right\|_h }_{\text{Spatial truncation error} \le C h^2}  + \left\| \muQb^n(\mathcal{I}_h u(t_n)) - \muQb^n(u_h^n) \right\|_h.
\end{align}
To strictly bound the discrete difference in the second term, we   apply the standard discrete cross-term decomposition to the nonlinear difference, which yields:
\begin{equation} \label{eq_cross_term_muQb}
    \mathcal{I}_h u(t_n) D_h^2 \mathcal{I}_h u(t_n) - u_h^n D_h^2 u_h^n = \mathcal{I}_h u(t_n) D_h^2 e_u^n + e_u^n D_h^2 u_h^n.
\end{equation}
Since the trace operator is linear, this decomposition naturally applies to the trace component as well. By absorbing the dimensional constants into a generic constant $C$, and utilizing the uniform $L^\infty$ bounds for both the projected exact solution and the numerical solution, we obtain:
\begin{align} \label{eq_muQb_discrete_bound}
    &\left\| \muQb^n(\mathcal{I}_h u(t_n)) - \muQb^n(u_h^n) \right\|_h \notag \\
    &\quad \le C \left\| \mathcal{I}_h u(t_n) D_h^2 e_u^n + e_u^n D_h^2 u_h^n \right\|_h \notag \\
    &\quad \le C \left( \|\mathcal{I}_h u(t_n)\|_{\infty} \|D_h^2 e_u^n\|_h + \|D_h^2 u_h^n\|_{\infty} \|e_u^n\|_h \right) \notag \\
    &\quad \le C(M) \left( \|e_u^n\|_{H^2_h} + \|e_u^n\|_h \right) \le C(M) \|e_u^n\|_{H^2_h}.
\end{align}
Combining \eqref{eq_muQb_triangle} and \eqref{eq_muQb_discrete_bound}, we conclude that $\| \muQb^n - \muQbz(t_n) \|_h \le C(M) \|e_u^n\|_{H^2_h} + C h^2$. Thus, the Lipschitz continuity for $\muQ$ is established by combining the bounds for both the polynomial part and the derivative coupling part:
\begin{align}
    \| \muQ^n-\muQz(t_n)\|_h &\le \| \muQa^n-\muQaz(t_n)\|_h + \| \muQb^n - \muQbz(t_n) \|_h \notag \\
    &\le C(M) \left( \|\boldsymbol{E}_{\mathbf{Q}}^{n}\|_h + \|e_{u}^{n}\|_{H^2_h} + h^2 \right).
\end{align}

Next, we analyze the Lipschitz continuity for $\muu$. The nonlinear term $\muu$ is decomposed into the polynomial part and the coupling part as follows:
\begin{equation}
    \mu_h(\mathbf{Q}_h^n, u_h^n) =
    \underbrace{ -a u_h^n - b (u_h^n)^2 - c (u_h^n)^3 -2B_0 |M_h^n|^2 u_h^n}_{\mu_{1,h}^n}
     \underbrace{ -2B_0 D_h^2 u_h^n : M_h^n -2B_0 \nabla_h \cdot \nabla_h \cdot (M_h^n u_h^n)}_{\mu_{2,h}^n}.
\end{equation}
Taking the $l^2$ norm of the difference for $\muua$ along with Theorem \ref{th4_1} yields the following bound:
\begin{align}
        \| \muua^n-\muuaz(t_n)\| \le C \left( \|\boldsymbol{E}_{\mathbf{Q}}^{n}\| + \|e_{u}^{n}\| \right).
\end{align}
For the differential coupling part $\mu_{2,h}^n$, which contains the discrete Hessian and double divergence operators, a direct comparison between the continuous and numerical terms would conflate spatial and numerical errors. We rigorously isolate these errors by introducing the spatial grid restriction operator $\mathcal{I}_h$. Applying the triangle inequality, we separate the spatial truncation error from the purely discrete numerical difference:
\begin{align} \label{eq_mu2_triangle}
    \| \mu_2(t_n) - \mu_{2,h}^n(\mathbf{Q}_h^n, u_h^n) \|_h
    &\le \underbrace{ \| \mu_2(t_n) - \mu_{2,h}^n(\mathcal{I}_h \mathbf{Q}(t_n), \mathcal{I}_h u(t_n)) \|_h }_{\text{Spatial truncation error } \|\boldsymbol{\tau}_{\mu_2}^n\|_h \le C h^2} \notag \\
    &\quad + \left\| \mu_{2,h}^n(\mathcal{I}_h \mathbf{Q}(t_n), \mathcal{I}_h u(t_n)) - \mu_{2,h}^n(\mathbf{Q}_h^n, u_h^n) \right\|_h.
\end{align}
To bound the discrete numerical error (the second term in \eqref{eq_mu2_triangle}), we utilize the discrete equivalences of the global errors, $e_u^n \equiv \mathcal{I}_h u(t_n) - u_h^n$ and $\boldsymbol{E}_{\mathbf{Q}}^n \equiv \mathcal{I}_h \mathbf{Q}(t_n) - \mathbf{Q}_h^n$. Since the tensor $M$ depends on $\mathbf{Q}$, its discrete difference is bounded by the order of $\boldsymbol{E}_{\mathbf{Q}}^n$, i.e., $\|\mathcal{I}_h M(t_n) - M_h^n\| \le C \|\boldsymbol{E}_{\mathbf{Q}}^n\|$.

We evaluate the discrete difference by separating it into the Hessian contraction part and the double divergence part. For the Hessian contraction, standard discrete cross-term decomposition yields:
\begin{align} \label{eq_mu2_hessian_bound}
    & 2B_0 \| D_h^2 \mathcal{I}_h u(t_n) : \mathcal{I}_h M(t_n) - D_h^2 u_h^n : M_h^n \|_h \notag \\
    &\quad = 2B_0 \| D_h^2 e_u^n : \mathcal{I}_h M(t_n) + D_h^2 u_h^n : (\mathcal{I}_h M(t_n) - M_h^n) \|_h \notag \\
    &\quad \le C \left( \|\mathcal{I}_h M(t_n)\|_{\infty} \|D_h^2 e_u^n\|_h + \|D_h^2 u_h^n\|_{\infty} \|\boldsymbol{E}_{\mathbf{Q}}^n\|_h \right) \notag \\
    &\quad \le C(M) \left( \|e_u^n\|_{H^2_h} + \|\boldsymbol{E}_{\mathbf{Q}}^n\|_h \right).
\end{align}

For the double divergence part, we apply the discrete product rule expansion to both the projected exact grid functions and the numerical solutions:
\begin{equation} \label{eq_discrete_product_rule}
    \nabla_h \cdot \nabla_h \cdot (M_h^n u_h^n) = (\nabla_h \cdot \nabla_h \cdot M_h^n) u_h^n + 2 (\nabla_h \cdot M_h^n) \cdot \nabla_h u_h^n + M_h^n : D_h^2 u_h^n.
\end{equation}
By matching the corresponding expanded terms and applying Hölder's inequality alongside the uniform $W^{2,\infty}$ bounds, we derive the discrete estimates for the three sub-components:
\begin{equation} \label{eq_mu2_div_bounds}
\begin{aligned}
    \| [(\nabla_h \cdot \nabla_h \cdot \mathcal{I}_h M(t_n)) \mathcal{I}_h u(t_n)] - [(\nabla_h \cdot \nabla_h \cdot M_h^n) u_h^n] \|_h
    &\le C \left( \|\mathcal{I}_h u(t_n)\|_{\infty} \|\boldsymbol{E}_{\mathbf{Q}}^{n}\|_{H^2_h} + \|\nabla_h^2 M_h^n\|_{\infty} \|e_{u}^{n}\|_h \right), \\
    2 \| [(\nabla_h \cdot \mathcal{I}_h M(t_n)) \cdot \nabla_h \mathcal{I}_h u(t_n)] - [(\nabla_h \cdot M_h^n) \cdot \nabla_h u_h^n] \|_h
    &\le C \left( \|\nabla_h \mathcal{I}_h u(t_n)\|_{\infty} \|\boldsymbol{E}_{\mathbf{Q}}^{n}\|_{H^1_h} + \|\nabla_h M_h^n\|_{\infty} \|e_{u}^{n}\|_{H^1_h} \right), \\
    \| [\mathcal{I}_h M(t_n) : D_h^2 \mathcal{I}_h u(t_n)] - [M_h^n : D_h^2 u_h^n] \|_h
    &\le C \left( \|\mathcal{I}_h M(t_n)\|_{\infty} \|e_{u}^{n}\|_{H^2_h} + \|D_h^2 u_h^n\|_{\infty} \|\boldsymbol{E}_{\mathbf{Q}}^{n}\|_h \right).
\end{aligned}
\end{equation}
Summing the sub-components in \eqref{eq_mu2_div_bounds} reveals that the dominant terms require second-order spatial derivatives of the errors. Combining this with the Hessian bound \eqref{eq_mu2_hessian_bound}, the total discrete numerical error for $\mu_{2,h}^n$ is bounded by:
\begin{equation} \label{eq_mu2_total_discrete}
    \left\| \mu_{2,h}^n(\mathcal{I}_h \mathbf{Q}(t_n), \mathcal{I}_h u(t_n)) - \mu_{2,h}^n(\mathbf{Q}_h^n, u_h^n) \right\|_h \le C(M) \left( \|\boldsymbol{E}_{\mathbf{Q}}^n\|_{H^2_h} + \|e_u^n\|_{H^2_h} \right).
\end{equation}
Combining the spatial truncation error and the discrete numerical error, we conclude that $$\| \mu_2(t_n) - \mu_{2,h}^n(\mathbf{Q}_h^n, u_h^n) \|_h \le C(M) \left( \|\boldsymbol{E}_{\mathbf{Q}}^n\|_{H^2_h} + \|e_u^n\|_{H^2_h} + h^2 \right)$$.
Finally, by combining the bounds for both the polynomial part and the coupling part, we establish the Lipschitz continuity for $\muu$:
\begin{align}
    \| \muu^n-\muuz(t_n)\|_h &\le \| \mu_{1,h}^n-\mu_{1}(t_n)\|_h + \| \mu_{2,h}^n- \mu_{2}(t_n) \|_h \notag \\
    &\le C(M) \left( \|\boldsymbol{E}_{\mathbf{Q}}^{n}\|_{H^2_h} + \|e_{u}^{n}\|_{H^2_h} + h^2 \right).
\end{align}
\end{proof}

\begin{theorem} \label{thm:full_lipschitz}
    With the Lipschitz continuity of the nonlinear terms established in Lemma~\ref{lem:Lipschitz} and the error estimates from Theorem \ref{th4_1}, we can rigorously derive the following bound for the multiplier error:
    \begin{equation}
    |g^n - g(t_n)| \le C \left( |e_s^n| + \|\boldsymbol{E}_{\mathbf{Q}}^n\|_{H^1_h} + \|e_u^{n}\|_{H^2_h} + h^2 \right).
\end{equation}
Combining Lemma~\ref{lem:Lipschitz} and Theorem \ref{th4_1}, we obtain the following combined estimates:
\begin{align}
    \| \mNQ^n -\mNQz(t_n)\| \le C_{Lip, \mathbf{Q}} \left( |e_s^n| + \|\boldsymbol{E}_{\mathbf{Q}}^{n}\|_{H^1} + \|e_{u}^{n}\|_{H^2}+ h^2 \right),\label{eq_full_Lip_Q}\\
    \|  \mNu^n- \mNuz(t_n)\| \le C_{Lip, u} \left( |e_s^n| + \|\boldsymbol{E}_{\mathbf{Q}}^{n}\|_{H^2} + \|e_{u}^{n}\|_{H^2}+ h^2 \right),\label{eq_full_Lip_Q2}
\end{align}
where  $C_{Lip, \mathbf{Q}}$ and $C_{Lip, u}$ are positive constants depending only on $M$, $\Omega$, and the exact solution.
\end{theorem}
\begin{proof}
 By the definitions of the numerical multiplier $g^n = \exp(s^n)/\exp(E_{1h}(\mathbf{Q}_h^n, u_h^n))$ and the exact continuous multiplier $g(t_n) = \exp(s(t_n))/\exp(E_1(\mathbf{Q}(t_n), u(t_n)))$, we can decompose the error $|g^n - g(t_n)|$ into two primary parts using the triangle inequality:
\begin{align}
    |g^n - g(t_n)| &\le \underbrace{ \left| \frac{\exp(s^n)}{\exp(E_{1h}(\mathbf{Q}_h^n, u_h^n))} - \frac{\exp(s(t_n))}{\exp(E_{1h}(\mathbf{Q}_h^n, u_h^n))} \right| }_{\mathcal{T}_1} \notag \\
    &\quad + \underbrace{ \left| \frac{\exp(s(t_n))}{\exp(E_{1h}(\mathbf{Q}_h^n, u_h^n))} - \frac{\exp(s(t_n))}{\exp(E_1(\mathbf{Q}(t_n), u(t_n)))} \right| }_{\mathcal{T}_2}. \label{eq_g_diff_split_2step}
\end{align}

By the Mean Value Theorem and the uniform boundedness of the solutions, we first estimate $\mathcal{T}_1$. Fixing the denominator yields
\begin{equation*}
    \mathcal{T}_1 \le \frac{\exp(\xi^n)}{\exp(E_{1h}(\mathbf{Q}_h^n, u_h^n))} |s^n - s(t_n)| \le C |e_s^n|,
\end{equation*}
with $\xi^n$ being a number between $s^n$ and $s(t_n)$.

For $\mathcal{T}_2$, applying the Mean Value Theorem to the exponential function $f(x) = e^{-x}$ gives
\begin{align}
    \mathcal{T}_2 &= \exp(s(t_n)) \left| \exp(-E_{1h}(\mathbf{Q}_h^n, u_h^n)) - \exp(-E_1(\mathbf{Q}(t_n), u(t_n))) \right| \notag \\
    &\le \exp(s(t_n)) \exp(-\eta^n) \left| E_1(\mathbf{Q}(t_n), u(t_n)) - E_{1h}(\mathbf{Q}_h^n, u_h^n) \right| \notag \\
    &\le C \left| E_1(\mathbf{Q}(t_n), u(t_n)) - E_{1h}(\mathbf{Q}_h^n, u_h^n) \right|, \label{eq_T2_energy_diff}
\end{align}
where $\eta^n$ is a number between $E_{1h}(\mathbf{Q}_h^n, u_h^n)$ and $E_1(\mathbf{Q}(t_n), u(t_n))$.

By seamlessly inserting the discrete energy functional evaluated at these interpolated exact solutions, namely $E_{1h}(\mathcal{I}_h \mathbf{Q}(t_n), \mathcal{I}_h u(t_n))$, we can rigorously decompose the total energy difference into the spatial truncation error and the purely numerical error:
\begin{align}
    \left| E_1(\mathbf{Q}(t_n), u(t_n)) - E_{1h}(\mathbf{Q}_h^n, u_h^n) \right| &\le \left| E_1(\mathbf{Q}(t_n), u(t_n)) - E_{1h}(\mathcal{I}_h \mathbf{Q}(t_n), \mathcal{I}_h u(t_n)) \right| \notag \\
    &\quad + \left| E_{1h}(\mathcal{I}_h \mathbf{Q}(t_n), \mathcal{I}_h u(t_n)) - E_{1h}(\mathbf{Q}_h^n, u_h^n) \right|. \label{eq_energy_sub_split}
\end{align}
The first term on the right-hand side of \eqref{eq_energy_sub_split} explicitly characterizes the spatial truncation error of the energy functional. Since the discrete functional $E_{1h}$ safely acts on the properly interpolated grid functions, this difference originates solely from the approximation of continuous integrals and differential operators by their discrete counterparts, which satisfies standard consistency properties:
\ba \label{eq_truncation_bound}
    \left| E_1(\mathbf{Q}(t_n), u(t_n)) - E_{1h}(\mathcal{I}_h \mathbf{Q}(t_n), \mathcal{I}_h u(t_n)) \right| \le C h^2.
\ed
The second term represents the purely numerical error on the discrete level. Since the variable differences correspond exactly to the aforementioned discrete numerical errors $\boldsymbol{E}_{\mathbf{Q}}^n$ and $e_u^n$, we can legally apply the algebraic identities, the Cauchy-Schwarz inequality, Theorem \ref{th4_1}, and utilize the previously derived cross-term decomposition strategy to obtain:
\ba \label{eq_numerical_bound}
    \left| E_{1h}(\mathcal{I}_h \mathbf{Q}(t_n), \mathcal{I}_h u(t_n)) - E_{1h}(\mathbf{Q}_h^n, u_h^n) \right| &\le C( \|\mathcal{I}_h \mathbf{Q}(t_n) - \mathbf{Q}_h^n\|_{H^1_h} + \|\mathcal{I}_h u(t_n) - u_h^n\|_{H^2_h} )\\ &\le C( \| \mathbf{Q}(t_n) - \mathbf{Q}_h^n\|_{H^1_h} + \| u(t_n) - u_h^n\|_{H^2_h} )\\ &\le C \left( \|\boldsymbol{E}_{\mathbf{Q}}^n\|_{H^1_h} + \|e_u^n\|_{H^2_h} \right).
\ed

Substituting \eqref{eq_truncation_bound} and \eqref{eq_numerical_bound} into \eqref{eq_energy_sub_split}, and subsequently incorporating the bounds for $\mathcal{T}_1$ and $\mathcal{T}_2$ back into \eqref{eq_g_diff_split_2step}, we finally arrive at
\begin{equation}\label{eq_14}
    |g^n - g(t_n)| \le C \left( |e_s^n| + \|\boldsymbol{E}_{\mathbf{Q}}^n\|_{H^1_h} + \|e_u^{n}\|_{H^2_h} + h^2 \right).
\end{equation}

Using the algebraic identity $g^n \muQ^n - g(t_n) \muQz(t_n) =  g^n (\muQ^n - \muQz(t_n))+(g^n - g(t_n)) \muQz(t_n) $, along with the Lemma~\ref{lem:Lipschitz} and the estimates \eqref{eq_14}, we can estimate the discrete $l^2$-norm of the product difference as follows:
\begin{equation} \label{eq_J_Q_bound}
\begin{aligned}
    \| g^n \muQ^n - g(t_n) \muQz(t_n) \|_h
    &\le \| \muQz(t_n) \|_h |g^n - g(t_n)| + |g^n| \|\muQ^n- \muQz(t_n)\|_h \\
    &\le  C \left( |e_s^n| + \|\boldsymbol{E}_{\mathbf{Q}}^n\|_{H^1_h} + \|e_u^{n}\|_{H^2_h} +h^2\right) + C \left( \|\boldsymbol{E}_{\mathbf{Q}}^{n}\|_{h} + \|e_{u}^{n}\|_{H^2_h} + h^2 \right) \\
    &\le C \left( |e_s^n| + \|\boldsymbol{E}_{\mathbf{Q}}^{n}\|_{H^1_h} + \|e_{u}^{n}\|_{H^2_h}+h^2 \right).
\end{aligned}
\end{equation}
Similarly, we can also estimate the product difference for $\muu$ as follows:
\begin{equation} \label{eq_J_u_bound}
\begin{aligned}
    \| g^n \muu^n - g(t_n) \muuz(t_n) \|_h
    &\le \| \muuz(t_n) \|_h |g^n - g(t_n)| + |g^n| \|  \muu^n- \muuz(t_n)\|_h \\
    &\le  C \left( |e_s^n| + \|\boldsymbol{E}_{\mathbf{Q}}^n\|_{H^1_h} + \|e_u^{n}\|_{H^2_h} +h^2\right) + C \left( \|\boldsymbol{E}_{\mathbf{Q}}^{n}\|_{H^2_h} + \|e_{u}^{n}\|_{H^2_h} + h^2 \right) \\
    &\le C \left( |e_s^n| + \|\boldsymbol{E}_{\mathbf{Q}}^{n}\|_{H^2_h} + \|e_{u}^{n}\|_{H^2_h}+h^2 \right).
\end{aligned}
\end{equation}
Combining the bounds \eqref{eq_J_Q_bound} and \eqref{eq_J_u_bound} yields the desired conclusion as stated in \eqref{eq_full_Lip_Q} and \eqref{eq_full_Lip_Q2}.
\end{proof}
\begin{theorem}[Fully Discrete Error Estimate] \label{thm:global_error}
    Let $(\mathbf{Q}(t), u(t), s(t))$ be the exact solution to the coupled system, satisfying the regularity assumptions $u, \mathbf{Q} \in L^\infty(0,T; H^{k+2}(\Omega))$ and $\partial_t u, \partial_t \mathbf{Q} \in L^\infty(0,T; H^{k+2}(\Omega))$, where $k$ is the order of the spatial discretization.

    Assume that the time step $\tau$ and mesh size $h$ are sufficiently small. Then, the fully discrete numerical solution $(\mathbf{Q}_h^n, u_h^n, s^n)$ obtained by the ETD1 scheme satisfies the following error estimate for all $n$ such that $n\tau \le T$:
    \begin{equation} \label{eq_final_total_error}
        \| \boldsymbol{E}_{\mathbf{Q}}^{n} \|_{H^1_h} + \| e_{u}^{n} \|_{H^2_h} + |e_s^n| \le C (\tau + h^2),
    \end{equation}
    where $C$ is a positive constant independent of $\tau$, $h$, and $n$.
\end{theorem}
\begin{proof}
    Under the standard induction hypothesis $\|\boldsymbol{E}_{\mathbf{Q}}^n\|_{H_h^1}^2+\ib{e_u^{n}} \le \frac{1}{2}$ and the temporary bootstrap assumption $\|\boldsymbol{E}_{\mathbf{Q}}^{n+1}\|_{H_h^1}^2+\ib{e_u^{n+1}} \le \frac{1}{2}$, the norm of the error increment is naturally bounded by a generic constant via the triangle inequality:
\ba\label{eq_error_increment_bound}
   \ibQ{\boldsymbol{E}_{\mathbf{Q}}^{n+1} - \boldsymbol{E}_{\mathbf{Q}}}&\le  \|\boldsymbol{E}_{\mathbf{Q}}^{n+1} - \boldsymbol{E}_{\mathbf{Q}}^n\|_{H_h^1}^2 \le \|\boldsymbol{E}_{\mathbf{Q}}^{n+1}\|_{H_h^1}^2 + \|\boldsymbol{E}_{\mathbf{Q}}^n\|_{H_h^1}^2 \le 1,\\
    \ibu{e_u^{n+1} - e_u^n} &\le \|e_u^{n+1} - e_u^n\|_{H_h^2}^2 \le \|e_u^{n+1}\|_{H_h^2}^2 + \|e_u^n\|_{H_h^2}^2 \le 1.
\ed

Taking the discrete inner product of (\ref{eq_error_s_final}) with $\ea$ and (\ref{eq_error_s_final2}) with $\eb$, and summing them up, we obtain the global error energy equation
\begin{align*}
        &\left\langle  \mathcal{Q}(\tau \mLQ) \ea, \ea \right\rangle_h + \left\langle \mLQ \boldsymbol{E}_{\mathbf{Q}}^{n+1}, \ea \right\rangle_h
    + \left\langle  \mathcal{Q}(\tau \mLu) \eb, \eb \right\rangle_h + \left\langle \mLu e_{u}^{n+1}, \eb \right\rangle_h \\
    &= \left\langle   \mNQ^n -\mNQz(t_n) - \mathbf{T}_{\mathbf{Q}}^n, \ea \right\rangle_h + \left\langle \mathcal{N}_u^{n}-\mNuz(t_n)- \mathbf{T}_{u}^n, \eb \right\rangle_h.
\end{align*}
Applying a standard algebraic identity and the norm defined in \eqref{eq_inner_Q} to the left-hand side yields:
\begin{equation} \label{eq_LHS_expanded_norm}
\begin{aligned}
    \text{LHS} &= \iaQ{\ea} + \frac{1}{2\tau} \left[  \ilQ{\boldsymbol{E}_{\mathbf{Q}}^{n+1}} - \ilQ{\boldsymbol{E}_{\mathbf{Q}}^{n}}  + \left\langle \tau^2 \mLQ \ea, \ea \right\rangle_h \right] \\
    &\quad + \iau{\eb} + \frac{1}{2\tau} \left[ \ilu{e_{u}^{n+1}}  - \ilu{e_{u}^{n}}  +\left\langle \tau^2 \mLu \delta_t  e_{u}^{n+1}, \eb \right\rangle_h  \right] \\
    &= \ibQ{\ea} + \frac{1}{2\tau} \left[  \ilQ{\boldsymbol{E}_{\mathbf{Q}}^{n+1}} - \ilQ{\boldsymbol{E}_{\mathbf{Q}}^{n}}   \right]
    + \ibu{\eb} + \frac{1}{2\tau} \left[ \ilu{e_{u}^{n+1}}  - \ilu{e_{u}^{n}}  \right].
\end{aligned}
\end{equation}
For the right-hand side, applying Young's inequality with the parameter $\gamma = 1/8$, and subsequently utilizing the Lipschitz bounds from Theorem \ref{thm:full_lipschitz} alongside the truncation error bounds from Theorem \ref{th4_1}, we can bound the RHS as follows:
\begin{equation} \label{eq_RHS_final_bound}
\begin{aligned}
    \text{RHS} &\leq 2\gamma \left( \ia{\ea} + \ia{\eb} \right) + \frac{1}{4\gamma} \left( \ia{\mathcal{N}_u^{n}-\mNuz(t_n)-\mathbf{T}_{\mathbf{Q}}^n} + \ia{\mNQ^n -\mNQz(t_n)-\mathbf{T}_{u}^n} \right) \\
    &= \frac{1}{2} \left( \|\ea\|_h^2 + \|\eb\|_h^2 \right) \\
    &\quad+  \left( \|\mathcal{N}_u^{n}-\mNuz(t_n)\|_h^2  + \ia{\mNQ^n -\mNQz(t_n)} + \ia{\mathbf{T}_{\mathbf{Q}}^n} + \|\mathbf{T}_{u}^n\|_h^2 \right) \\
    &\leq \frac{1}{4} \left( \|\ea\|_h^2 + \|\eb\|_h^2 \right) +   C_{trunc} (\tau+h^2)^2 \\
    &\quad +   C_{Lip, \mathbf{Q}}^2 \left( |e_s^n|^2 + \|\boldsymbol{E}_{\mathbf{Q}}^{n}\|_{H_h^1}^2 + \|e_{u}^{n}\|_{H_h^2}^2 \right) + C_{Lip, u}^2 \left( |e_s^n|^2 + \|\boldsymbol{E}_{\mathbf{Q}}^{n}\|_{H_h^2}^2 + \|e_{u}^{n}\|_{H_h^2}^2 \right)  \\
    &\leq \frac{1}{4} \left( \|\ea\|_h^2 + \|\eb\|_h^2 \right)
     + C_{nonlin} \left( |e_s^n|^2 + \|\boldsymbol{E}_{\mathbf{Q}}^{n}\|_{H_h^2}^2 + \|e_{u}^{n}\|_{H_h^2}^2 \right)+ C_{trunc} (\tau+h^2)^2.
\end{aligned}
\end{equation}
Combining the estimates in \eqref{eq_LHS_expanded_norm} and \eqref{eq_RHS_final_bound}, we obtain the following error energy inequality:
\begin{equation} \label{eq_error_energy_recursive}
\begin{aligned}
    &\frac{3}{4} \left( \|\ea\|_h^2 + \|\eb\|_h^2 \right) + \frac{1}{2\tau} \left[  \ilQ{\boldsymbol{E}_{\mathbf{Q}}^{n+1}} - \ilQ{\boldsymbol{E}_{\mathbf{Q}}^{n}}   + \ilu{e_{u}^{n+1}}  - \ilu{e_{u}^{n}}  \right] \\
    &\leq C_{nonlin} \left( |e_s^n|^2 + \|\boldsymbol{E}_{\mathbf{Q}}^{n}\|_{H_h^2}^2 + \|e_{u}^{n}\|_{H_h^2}^2 \right)+ C_{trunc} (\tau+h^2)^2.
\end{aligned}
\end{equation}

Taking the discrete inner product of (\ref{eq_error_s_final}) with the test function $\mLQ \boldsymbol{E}_{\mathbf{Q}}^{n}$, we obtain the $H^1$ energy evolution equation with $H^2$ dissipation:
\begin{equation} \label{eq_H2_bound_shifted_detailed}
\begin{aligned}
    &\left\langle  \mQ \delta_t \boldsymbol{E}_{\mathbf{Q}}^{n}, \mLQ \boldsymbol{E}_{\mathbf{Q}}^{n} \right\rangle_h + \|\mLQ \boldsymbol{E}_{\mathbf{Q}}^{n}\|_h^2 = \left\langle \delta \mNQ^n - \mathbf{T}_{\mathbf{Q}}^{n-1}, \mLQ \boldsymbol{E}_{\mathbf{Q}}^{n} \right\rangle_h.
\end{aligned}
\end{equation}
Applying a standard algebraic identity to the first term on the left-hand side, along with Lemma \ref{lem_norm_equivalence}, yields:
\begin{equation}
\begin{aligned}
 \left\langle  \mQ\delta_t \boldsymbol{E}_{\mathbf{Q}}^{n}, \mLQ\boldsymbol{E}_{\mathbf{Q}}^{n} \right\rangle_h&= \left\langle  \mQ\mLQ\delta_t \boldsymbol{E}_{\mathbf{Q}}^{n}, \boldsymbol{E}_{\mathbf{Q}}^{n} \right\rangle_h
 \\&=\frac{1}{2\tau}\left(\left\langle \mQ \mLQ  \boldsymbol{E}_{\mathbf{Q}}^{n}, \boldsymbol{E}_{\mathbf{Q}}^{n} \right\rangle_h - \left\langle \mQ \mLQ  \boldsymbol{E}_{\mathbf{Q}}^{n-1}, \boldsymbol{E}_{\mathbf{Q}}^{n-1} \right\rangle_h\right.\\
 &\quad \left. + \left\langle \mQ \mLQ  (\boldsymbol{E}_{\mathbf{Q}}^{n}-\boldsymbol{E}_{\mathbf{Q}}^{n-1}), \boldsymbol{E}_{\mathbf{Q}}^{n}-\boldsymbol{E}_{\mathbf{Q}}^{n-1} \right\rangle_h\right)
  \\&=\frac{1}{2\tau} \left( \icQ{\boldsymbol{E}_{\mathbf{Q}}^{n}} - \icQ{\boldsymbol{E}_{\mathbf{Q}}^{n-1}} \right) + \frac{1}{2}\tau \left(  \left\langle \mathcal{Q} (\tau \mLQ)\mLQ  (\delta_t \boldsymbol{E}_{\mathbf{Q}}^{n}),\delta_t \boldsymbol{E}_{\mathbf{Q}}^{n}\right\rangle_h\right) \\
    &\geq \frac{1}{2\tau} \left( \icQ{\boldsymbol{E}_{\mathbf{Q}}^{n}} - \icQ{\boldsymbol{E}_{\mathbf{Q}}^{n-1}} \right).
\end{aligned}
\end{equation}
 Then the RHS can be estimated by applying Young's inequality with the same parameter $\gamma = \frac{1}{2}$, which yields:
 \begin{align}
     \left\langle \delta \mNQ^n - \mathbf{T}_{\mathbf{Q}}^{n-1}, \mLQ \boldsymbol{E}_{\mathbf{Q}}^{n} \right\rangle_h &\le \frac{1}{2} \|\delta \mNQ^n - \mathbf{T}_{\mathbf{Q}}^{n-1}\|_h^2 + \frac{1}{2} \|\mLQ \boldsymbol{E}_{\mathbf{Q}}^{n}\|_h^2\no\\&\le C \left( |e_s^{n-1}|^2 + \|\boldsymbol{E}_{\mathbf{Q}}^{n-1}\|_h^2 + \|e_{u}^{n-1}\|_{H_h^2}^2 \right) + C (\tau+h^2)^2+ \frac{1}{2} \|\mLQ \boldsymbol{E}_{\mathbf{Q}}^{n}\|_h^2.\label{eq_source_prev_step}
 \end{align}
Substituting (\ref{eq_source_prev_step}) back into (\ref{eq_H2_bound_shifted_detailed}), we have:
\begin{equation} \label{eq_H2_full_bound}
   \frac{C_1}{\tau} \left( \icQ{\boldsymbol{E}_{\mathbf{Q}}^{n}} - \icQ{\boldsymbol{E}_{\mathbf{Q}}^{n-1}} \right)+ \|\boldsymbol{E}_{\mathbf{Q}}^{n}\|_{H_h^2}^2 \le   C_2 \left( |e_s^{n-1}|^2 + \|\boldsymbol{E}_{\mathbf{Q}}^{n-1}\|_{H_h^1}^2 + \|e_{u}^{n-1}\|_{H_h^2}^2 \right) + C_3 (\tau+h^2)^2.
\end{equation}

Next, we consider the error equations \eqref{eq_error_s_final3}--\eqref{eq_error_s_final4} for the auxiliary variable $s$. Returning to \eqref{eq_error_s_final4}, we first rearrange the update formula for $s^{n+1}$. By subtracting the exact solution $s(t_{n+1})$, the corresponding error $e_s^{n+1}$ can be expressed as:
\begin{equation} \label{eq_es_derivation}
\begin{aligned}
    e_s^{n+1} &= s^{n+1} - s(t_{n+1}) \\
    &= \left[ \xi^{n+1} \tilde{s}^{n+1} + (1 - \xi^{n+1}) E_{1h}^{n+1} \right] - s(t_{n+1}) \\
    &= \tilde{s}^{n+1} - s(t_{n+1}) + (1 - \xi^{n+1})(E_{1h}^{n+1} - \tilde{s}^{n+1}) \\
    &= \tilde{e}_s^{n+1} + (1 - \xi^{n+1})(E_{1h}^{n+1} - \tilde{s}^{n+1}).
\end{aligned}
\end{equation}
Substituting the explicit expression for $\xi^{n+1}$ given in \eqref{eq_xi_optimal} yields the exact error relation for the auxiliary variable:
\begin{equation}
    e_s^{n+1} = \tilde{e}_s^{n+1} + \eta_0 \tau \mathcal{R}^{n+1}.
\end{equation}
To extract the desired $(1+\tau)$ amplification factor necessary for the application of the discrete Gronwall lemma, we directly apply the weighted Young's inequality in the form $(a+b)^2 \le (1+\tau)a^2 + (1+\tau^{-1})b^2$. This elegantly bounds the squared norm in a single cohesive step:
\begin{equation} \label{eq_es_intermediate_bound}
\begin{aligned}
    |e_s^{n+1}|^2 &\le (1+\tau)|\tilde{e}_s^{n+1}|^2 + \left(1+\frac{1}{\tau}\right) (\eta_0 \tau \mathcal{R}^{n+1})^2 \\
    &= (1+\tau)|\tilde{e}_s^{n+1}|^2 + \frac{1+\tau}{\tau} \eta_0^2 \tau^2 (\mathcal{R}^{n+1})^2 \\
    &= (1+\tau)|\tilde{e}_s^{n+1}|^2 + \eta_0^2 \tau (1+\tau) (\mathcal{R}^{n+1})^2.
\end{aligned}
\end{equation}

Recalling the definition of the discrete dissipation rate $\mathcal{R}^{n+1}$,  we apply the basic inequality $(a+b)^2 \le 2a^2 + 2b^2$ to separate the variables:
\begin{equation} \label{eq_nonlinear_split1}
   \eta_0^2  (\mathcal{R}^{n+1})^2 = \eta_0^2 \tau^2 \left( \ibQ{\delta_t \mathbf{Q}_h^{n+1}}+ \ibu{\delta_t u_h^{n+1}}\right)^2 \leq 2\eta_0^2 \tau^2 \|\delta_t \mathbf{Q}_h^{n+1}\|_h^4 + 2\eta_0^2 \tau^2 \|\delta_t u_h^{n+1}\|_h^4.
\end{equation}
Utilizing the weighted inequality $(a+b)^4 \le (1+\epsilon) b^4 + C_\epsilon a^4$ with the choice of $\epsilon = 1/2$, we can stringently bound the highly nonlinear terms while minimizing the coefficient accumulated on the error component. The expansion is given as follows:
\begin{equation} \label{eq_nonlinear_split2}
\begin{aligned}
    2\eta_0^2 \tau^2 \|\delta_t \mathbf{Q}_h^{n+1}\|_{\mathcal{Q}_\mathcal{L}^1}^4 &\leq C_\epsilon \eta_0^2 \tau^2 \|\delta_t \mathbf{Q}(t_{n+1})\|_{\mathcal{Q}_\mathcal{L}^1}^4 + 3\eta_0^2 \tau^2 \|\delta_t \eQ\|_{\mathcal{Q}_\mathcal{L}^1}^4, \\
    2\eta_0^2 \tau^2 \|\delta_t u_h^{n+1}\|_{\mathcal{Q}_\mathcal{D}^1}^4 &\leq C_\epsilon \eta_0^2 \tau^2 \|\delta_t u(t_{n+1})\|_{\mathcal{Q}_\mathcal{D}^1}^4 + 3\eta_0^2 \tau^2 \|\delta_t e_u^{n+1}\|_{\mathcal{Q}_\mathcal{D}^1}^4,
\end{aligned}
\end{equation}
where $C_\epsilon$ is a positive constant depending only on $\epsilon$.
By the smoothness of the exact solutions, the terms involving the exact time derivatives, $\|\delta_t \mathbf{Q}(t_{n+1})\|_{\mathcal{Q}_\mathcal{L}^1}^4$ and $\|\delta_t u(t_{n+1})\|_{\mathcal{Q}_\mathcal{D}^1}^4$, are bounded by a generic constant $C$ independent of $\tau$ and $h$.
For the error parts, applying the discrete Sobolev embedding inequality alongside the standard \textit{a priori} energy assumption (which dictates that the spatial gradients of the errors are sufficiently small), we can bound these terms and subsequently absorb them into the left-hand side. Specifically, on the basis of \eqref{eq_error_increment_bound},  we can obtain the following bounds for the error terms:
\begin{equation}\label{eq_error_nonlinear_bound}
\begin{aligned}
    3\eta_0^2\tau^2\|\delta_t \eQ\|_{\mathcal{Q}_\mathcal{L}^1}^4 &\leq 3\eta_0^2\|\delta \eQ\|_{\mathcal{Q}_\mathcal{L}^1}^2\|\delta_t \eQ\|_{\mathcal{Q}_\mathcal{L}^1}^2 \leq \frac{1}{4} \|\delta_t \eQ\|_{\mathcal{Q}_\mathcal{L}^1}^2, \\
    3\eta_0^2\tau^2\|\delta_t e_u^{n+1}\|_{\mathcal{Q}_\mathcal{D}^1}^4 &\leq 3\eta_0^2\|\delta e_u^{n+1}\|_{\mathcal{Q}_\mathcal{D}^1}^2\|\delta_t e_u^{n+1}\|_{\mathcal{Q}_\mathcal{D}^1}^2 \leq \frac{1}{4} \|\delta_t e_u^{n+1}\|_{\mathcal{Q}_\mathcal{D}^1}^2.
\end{aligned}
\end{equation}
Substituting the bounds from \eqref{eq_nonlinear_split1}, \eqref{eq_nonlinear_split2}, and \eqref{eq_error_nonlinear_bound} back into \eqref{eq_es_intermediate_bound}, we arrive at the following key estimate for the auxiliary variable error:
\begin{equation}\label{eq_sub_corr}
    |e_s^{n+1}|^2 \le (1 + \tau) |\tilde{e}_s^{n+1}|^2 + \tau(1 + \tau) (3\eta_0^2 \tau^2 C + \frac{1}{4} \|\delta_t \boldsymbol{e}_{\mathbf{Q}}^{n+1}\|_h^2 + \frac{1}{4} \|\delta_t e_u^{n+1}\|_h^2).
\end{equation}

Next, we return to the error equation \eqref{eq_error_s_final3} for $\tilde{s}$.
Multiplying the error equation \eqref{eq_error_s_final3} for $\tilde{s}$  by $2\tilde{e}_s^{n+1}$ yields the following equation:
\begin{equation} \label{eq_es_identity}
\begin{aligned}
    &\frac{1}{\tau} \left( |\tilde{e}_s^{n+1}|^2 - |e_s^n|^2 + |\tilde{e}_s^{n+1} - e_s^n|^2 \right)
    = \underbrace{2 \tilde{e}_s^{n+1} \Big[ g^n \langle \muQ^n, \ea \rangle_h + g^n \langle \mu_h^n, \eb \rangle_h \Big]}_{\mathcal{T}_1}\\
    & \qquad+ \underbrace{2\tilde{e}_s^{n+1} \Big[ \left\langle g^n \muQ^n-\muQz(t_n), \da  \right\rangle_h + \left\langle g^n \muu^n-\muuz(t_n), \db \right\rangle_h -R_{1s}^n  \Big]}_{\mathcal{T}_2}.
\end{aligned}
\end{equation}
Using Theorem \ref{th4_1} and Young's inequality, we can bound the term $\mathcal{T}_1$ in \eqref{eq_es_identity} as follows:
\begin{equation} \label{eq_I2_bound}
\begin{aligned}
   \mathcal{T}_1&\le 2 |\tilde{e}_s^{n+1}| \cdot |g^n| \cdot \left( C_\mu \|\ea\|_h + C_\mu' \|\eb\|_h \right) \\
    &\le 2 |\tilde{e}_s^{n+1}| \cdot C_g \cdot C_\mu \|\ea\|_h + 2 |\tilde{e}_s^{n+1}| \cdot C_g \cdot C_\mu' \|\eb\|_h  \\
    &\le ( 8C_g^2 C_\mu^2   +8C_g^2 (C_\mu')^2) |\tilde{e}_s^{n+1}|^2 + \frac{1}{8} \|\ea\|_h^2 + \frac{1}{8}  \|\eb\|_h^2 .
\end{aligned}
\end{equation}
Using Theorem \ref{thm:full_lipschitz} and the temporal regularity of the exact solution, we can bound the term $\mathcal{T}_2$ in \eqref{eq_es_identity} as follows:
\begin{equation} \label{eq_I1_bound}
\begin{aligned}
    \mathcal{T}_2&\le  2|\tilde{e}_s^{n+1}|\cdot   \left(C_{Lip, \mathbf{Q}} \left( |e_s^n| + \|\boldsymbol{E}_{\mathbf{Q}}^{n}\|_{H_h^1} + \|e_{u}^{n}\|_{H_h^2}+ h^2 \right)\right.\\
     &\qquad \left.
    +C_{Lip, u} \left( |e_s^n| + \|\boldsymbol{E}_{\mathbf{Q}}^{n}\|_{H_h^2} + \|e_{u}^{n}\|_{H_h^2}+ h^2 \right) +C(\tau+h^2) \right)  \\
    &\le C \left( |\tilde{e}_s^{n+1}|^2 + |e_s^n|^2 + \|\boldsymbol{E}_{\mathbf{Q}}^{n}\|_{H_h^1}^2 + \|\boldsymbol{E}_{\mathbf{Q}}^{n}\|_{H_h^2}^2+ \|e_{u}^{n}\|_{H_h^2}^2 + (\tau+h^2)^2  \right).
\end{aligned}
\end{equation}
Substituting the bounds \eqref{eq_I1_bound} and \eqref{eq_I2_bound} back into the identity \eqref{eq_es_identity}, we arrive at the following overall estimate for the evolution of the error $e_s$:
\begin{equation} \label{eq_es_final_estimate}
\begin{aligned}
    \frac{1}{\tau} &\left( |\tilde{e}_s^{n+1}|^2 - |e_s^n|^2 \right) \le C \left( |\tilde{e}_s^{n+1}|^2 + |e_s^n|^2  \right) \\
    &\qquad + \underbrace{C \left( \|\boldsymbol{E}_{\mathbf{Q}}^{n}\|_{H_h^1}^2 + \|\boldsymbol{E}_{\mathbf{Q}}^{n}\|_{H_h^2}^2+ \|e_{u}^{n}\|_{H_h^2}^2 + (\tau + h^2)^2 \right) + \frac{1}{8} \left( \|\ea\|_h^2 + \|\eb\|_h^2 \right)}_{\mathcal{R}_{pred}}.
\end{aligned}
\end{equation}

Multiplying the predictor error equation \eqref{eq_es_final_estimate} by $\tau$ and rearranging the terms, we isolate $|\tilde{e}_s^{n+1}|^2$ on the left-hand side:
\begin{equation} \label{eq_pred_isolate}
    (1 - C\tau) |\tilde{e}_s^{n+1}|^2 \le (1 + C\tau) |e_s^n|^2 + \tau \mathcal{R}_{pred}.
\end{equation}
Under the mild time-step restriction $C\tau \le 1/2$, we can divide by $(1-C\tau)$. Utilizing the basic expansion $(1-C\tau)^{-1} \le 1 + 2C\tau$, we obtain an explicit upper bound for the intermediate error:
\ba \label{eq_sub_pred}
    |\tilde{e}_s^{n+1}|^2 &\le (1 + C\tau) (1 + 2C\tau) |e_s^n|^2 + \tau (1 + 2C\tau) \mathcal{R}_{pred}\\
    &\le (1 + C'\tau) |e_s^n|^2 + \tau (1+2C\tau) \mathcal{R}_{pred},
\ed
Substituting this intermediate estimate \eqref{eq_sub_pred} into \eqref{eq_sub_corr} leads to the combined bound:
\ba \label{eq_sub_combined}
    |e_s^{n+1}|^2 &\le (1 + \tau)(1 + C'\tau) |e_s^n|^2 + \tau (1+\tau)(1+2C\tau) \mathcal{R}_{pred} \\&+ 2\tau \left( 3\eta_0^2 \tau^2 C + \frac{1}{4} \|\delta_t \boldsymbol{e}_{\mathbf{Q}}^{n+1}\|_h^2 + \frac{1}{4} \|\delta_t e_u^{n+1}\|_h^2 \right)\\
    & \le(1 + C_s\tau) |e_s^n|^2+ 2\tau \left( 3\eta_0^2 \tau^2 C + \frac{1}{4} \|\delta_t \boldsymbol{e}_{\mathbf{Q}}^{n+1}\|_h^2 + \frac{1}{4} \|\delta_t e_u^{n+1}\|_h^2 \right)\\
    &+\tau(1+C'\tau)\left(C \left( \|\boldsymbol{E}_{\mathbf{Q}}^{n}\|_{H_h^1}^2 + \|\boldsymbol{E}_{\mathbf{Q}}^{n}\|_{H_h^2}^2+ \|e_{u}^{n}\|_{H_h^2}^2 + (\tau + h^2)^2 \right) + \frac{1}{8} \left( \|\ea\|_h^2 + \|\eb\|_h^2 \right)\right),\\
    &\le (1 + C_s\tau) |e_s^n|^2+C \tau\left( \|\boldsymbol{E}_{\mathbf{Q}}^{n}\|_{H_h^1}^2 + \|\boldsymbol{E}_{\mathbf{Q}}^{n}\|_{H_h^2}^2+ \|e_{u}^{n}\|_{H_h^2}^2 + (\tau + h^2)^2 \right) +\frac{1}{4} \tau \left( \|\ea\|_h^2 + \|\eb\|_h^2 \right)
\ed

By expanding the products, absorbing the resulting $\mathcal{O}(\tau)$ coefficients into a new generic constant $C_s$, subtracting $|e_s^n|^2$, and dividing by $\tau$, we recover the unified single-step error equation:
\begin{equation} \label{eq_unified_single}
    \frac{|e_s^{n+1}|^2 - |e_s^n|^2}{\tau} \le C_s |e_s^n|^2 + C \left( \|\boldsymbol{E}_{\mathbf{Q}}^{n}\|_{H_h^1}^2 + \|\boldsymbol{E}_{\mathbf{Q}}^{n}\|_{H_h^2}^2+ \|e_{u}^{n}\|_{H_h^2}^2 + (\tau + h^2)^2 \right) + \frac{1}{4} \left( \|\ea\|_h^2 + \|\eb\|_h^2 \right) .
\end{equation}

Combining \eqref{eq_unified_single} with the error energy inequality \eqref{eq_error_energy_recursive} and applying the $H_h^2$-norm bound \eqref{eq_H2_full_bound} yields the following overall energy inequality:
\begin{equation} \label{eq_error_evolution_shifted}
\begin{aligned}
    &\frac{1}{\tau} \left( |{e}_s^{n+1}|^2 - |e_s^n|^2 \right) + \frac{\tilde{C}_{mix}}{\tau} \left[ \icQ{\boldsymbol{E}_{\mathbf{Q}}^{n}} - \icQ{\boldsymbol{E}_{\mathbf{Q}}^{n-1}} \right] \\
    &+ \frac{1}{2\tau} \left[ \|\boldsymbol{E}_{\mathbf{Q}}^{n+1}\|_{\mathcal{L}_h}^2 - \|\boldsymbol{E}_{\mathbf{Q}}^{n}\|_{\mathcal{L}_h}^2 \right]
     + \frac{1}{2\tau} \left[ \|e_{u}^{n+1}\|_{\mathcal{D}_h}^2 - \|e_{u}^{n}\|_{\mathcal{D}_h}^2 \right]  \\
    &\leq C' (\tau + h^2)^2 + C' \left( |e_s^n|^2 + |e_s^{n-1}|^2 +\|\boldsymbol{E}_{\mathbf{Q}}^{n}\|_{H_h^1}^2+ \|\boldsymbol{E}_{\mathbf{Q}}^{n-1}\|_{H_h^1}^2 + \|e_{u}^{n}\|_{H_h^2}^2 + \|e_{u}^{n-1}\|_{H_h^2}^2 \right).
\end{aligned}
\end{equation}
To simplify the notation, we introduce the total discrete energy functional $E^n$ at time step $n$, which incorporates the corresponding norm coefficients:
\begin{equation*}
    E^n \coloneqq |e_s^n|^2 + \frac{1}{2}\|\boldsymbol{E}_{\mathbf{Q}}^{n}\|_{\mathcal{L}_h}^2 + \tilde{C}_{\text{mix}}\icQ{\boldsymbol{E}_{\mathbf{Q}}^{n-1}} + \frac{1}{2}\|e_{u}^{n}\|_{\mathcal{D}_h}^2.
\end{equation*}
Substituting this definition into the error evolution equation and utilizing the norm equivalences between the standard Sobolev norms and the weighted norms (i.e., $\|v\|_{H^1}^2 \le C \|v\|_{\mathcal{L}_h}^2$ and $\|v\|_{H^2}^2 \le C \|v\|_{\mathcal{D}_h}^2$), we deduce the following difference inequality:
\begin{equation*}
    \frac{E^{n+1} - E^n}{\tau} \le C (\tau + h^2)^2 + C (E^n + E^{n-1}),
\end{equation*}
where $C > 0$ is a generic constant independent of the time step $\tau$ and the spatial mesh size $h$.

Multiplying both sides by $\tau$ and summing over the time index $n$ from $1$ to $m-1$ (where $1 \le m \le N = T/\tau$), we obtain:
\begin{equation*}
    \sum_{n=1}^{m-1} (E^{n+1} - E^n) \le C \tau \sum_{n=1}^{m-1} (\tau + h^2)^2 + C \tau \sum_{n=1}^{m-1} (E^n + E^{n-1}).
\end{equation*}
The left-hand side is a telescoping sum that evaluates exactly to $E^m - E^1$. For the right-hand side, since $(m-1)\tau = t_{m-1} \le T$, the truncation error term is bounded by $C T (\tau + h^2)^2$. Furthermore, the summation of the energy terms can be neatly bounded as $\sum_{n=1}^{m-1} (E^n + E^{n-1}) \le 2 \sum_{k=0}^{m-1} E^k$. Consequently, the inequality simplifies to:
\begin{equation*}
    E^m \le E^1 + C T (\tau + h^2)^2 + 2C \tau \sum_{k=0}^{m-1} E^k.
\end{equation*}

Now, applying the standard discrete Gronwall lemma directly yields:
\begin{equation*}
    E^m \le \left( E^1 + C T (\tau + h^2)^2 \right) \exp\left( \sum_{k=0}^{m-1} 2C\tau \right) \le \left( E^1 + C T (\tau + h^2)^2 \right) e^{2CT}.
\end{equation*}

Provided that the initial error satisfies $E^0 \le C_0 (\tau + h^2)^2$, which is naturally guaranteed by standard spatial projections of the exact initial data. Furthermore, since the proposed numerical scheme is inherently a self-starting one-step method, the numerical solution at $t_1$ is directly computed from the initial state. Consequently, the error at the very first step, $E^1$, is strictly governed by the local truncation error of a single time iteration. Given the formal first-order temporal accuracy of the scheme, it trivially satisfies $E^1 \le C_0 (\tau + h^2)^2$ without mandating any specialized starting procedures.
Provided that the initial errors satisfy $E^0 \le C_0 (\tau + h^2)^2$ and $E^1 \le C_0 (\tau + h^2)^2$ (which can be achieved via a sufficiently accurate initialization scheme), we arrive at the final optimal error estimate:
\begin{equation*}
    E^m \le \tilde{C} (\tau + h^2)^2,
\end{equation*}
where $\tilde{C} > 0$ is a constant depending on $T$, $C$, and the initial data $C_0$, but strictly independent of $\tau$ and $h$. This completes the proof.
\end{proof}

\section{Numerical experiments}\label{section6}
 In this section, we present several numerical experiments to validate the MBP preservation, energy dissipation and convergence properties of the fully discrete SAV-EI scheme \ref{eq4_6} for the smectic-A liquid crystal model.
 The computational domain is taken as $\Omega=(0,2 \pi)^{3}$.
 \subsection{Convergence, MBP, and energy dissipation tests}\label{dim2}
In this subsection, the MBP, energy dissipation and convergence orders of the  proposed SAV-EI schemes are verified through numerical experiments.
  We take $N=128$ Fourier modes in each direction.
The initial conditions are set to be
\begin{align*}
    Q_0(x, y) &= \mathbf{n}\mathbf{n}^T - \frac{1}{2} I
    = \begin{pmatrix}
        \cos^2(x+y) - \frac{1}{2} & \cos(x+y)\sin(x+y) \\
        \cos(x+y)\sin(x+y) & \sin^2(x+y) - \frac{1}{2}
    \end{pmatrix}, \\
    u_0(x, y) &= 0.25 \cos(2\pi q x).
\end{align*}
 In our numerical simulations, the state of the smectic liquid crystal is primarily governed by the absolute temperature $T$. Following standard nondimensionalization, the rescaled temperatures are defined relative to the characteristic critical temperatures: $A = T - T_1^*$ for the isotropic-nematic transition and $a = T - T_2^*$ for the nematic-smectic transition. By setting the system temperature to $T = -1$, and the critical temperatures to $T_1^* = 0$ and $T_2^* = 4$, we obtain $A = -1$ and $a = -5$, ensuring the system is deeply quenched into the stable smectic phase. The remaining dimension-independent parameters are set as: elastic constant $K = 0.1$, bulk constant $C = 2.0$, nonlinear coefficients $b = 0$ and $c = 5$, characteristic wave number $q = 5$, coupling strength $B_0 = 0.7 \times 10^{-4}$, and mobility coefficients $\kappa_Q = \kappa_u = 8$.Due to the distinct algebraic properties of the $Q$-tensor in different spatial dimensions, the reference nematic scalar order parameter $s_+$ is defined separately for the 2D and 3D cases:For two-dimensional (2D) simulations ($d=2$), the system typically undergoes a continuous phase transition. Under the physical condition $A < 0$, the equilibrium order parameter is analytically given by:$$s_+ = \sqrt{\frac{-2A}{C}}.$$

  \textbf{Convergence tests.}
  This test verifies the convergence orders of the SAV-EI scheme in terms of both $\infty$-norm, the $l^2$-norm and the $H^1$-norm for the $Q$-tensor, as well as the $\infty$-norm, the $l^2$-norm and the $H^2$-norm for the velocity field $\bu$.
We set the final time to $T=1$ and calculate the numerical solution with a time step  size of $\tau=2^{-k}\tau_{1}$, where $k=0,1,\ldots,7$ and $\tau_{1}=2^{-7}$. For the lack of the exact solution, the numerical solution generated with $\tau=2^{-8}\tau_{1}$ at $T=1$ is regarded as the benchmark solution. The errors and convergence rates are presented in Table \ref{tab:error_Q} for the $Q$-tensor and Table \ref{tab:error_u_s} for the velocity field $\bu$ and the scalar variable $s$. The results confirm that the proposed SAV-EI scheme achieves first-order temporal accuracy in all norms, which is consistent with the theoretical analysis.
\begin{table}[htbp]
  \centering
  \setlength{\tabcolsep}{8pt}
  \begin{tabular}{l ccc}
  \toprule
  \multirow{2}{*}{$\tau$} & \multicolumn{3}{c}{$Q$} \\
  \cmidrule(lr){2-4}
   & $\infty$-norm & $l^2$-norm & $H^1$-norm \\
  \midrule
  $2^{-8}$  & 8.30e-3 (--)   & 3.62e-3 (--)   & 3.90e-2 (--)   \\
  $2^{-9}$  & 3.79e-3 (1.13) & 1.56e-3 (1.21) & 1.70e-2 (1.20) \\
  $2^{-10}$ & 1.60e-3 (1.24) & 6.72e-4 (1.22) & 7.35e-3 (1.21) \\
  $2^{-11}$ & 6.87e-4 (1.22) & 3.03e-4 (1.15) & 3.30e-3 (1.16) \\
  $2^{-12}$ & 3.12e-4 (1.14) & 1.43e-4 (1.09) & 1.54e-3 (1.10) \\
  $2^{-13}$ & 1.48e-4 (1.08) & 6.92e-5 (1.04) & 7.43e-4 (1.05) \\
  $2^{-14}$ & 7.20e-5 (1.04) & 3.40e-5 (1.02) & 3.64e-4 (1.03) \\
  $2^{-15}$ & 3.55e-5 (1.02) & 1.69e-5 (1.01) & 1.80e-4 (1.01) \\
  $2^{-16}$ & 1.76e-5 (1.01) & 8.41e-6 (1.01) & 8.98e-5 (1.01) \\
  $2^{-17}$ & 8.78e-6 (1.01) & 4.20e-6 (1.00) & 4.48e-5 (1.00) \\
  \bottomrule
  \end{tabular}
  \caption{Errors and convergence rates for the tensor field $Q$.}
  \label{tab:error_Q}
\end{table}

\begin{table}[htbp]
  \centering
  \setlength{\tabcolsep}{8pt}
  \begin{tabular}{l ccc c}
  \toprule
  \multirow{2}{*}{$\tau$} & \multicolumn{3}{c}{$u$} & $s$ \\
  \cmidrule(lr){2-4} \cmidrule(lr){5-5}
   & $\infty$-norm & $l^2$-norm & $H^2$-norm & $l^2$-norm \\
  \midrule
  $2^{-8}$  & 1.25e-1 (--)   & 6.92e-2 (--)   & 1.85e+0 (--)   & 6.46e-6 (--) \\
  $2^{-9}$  & 6.86e-2 (0.86) & 3.78e-2 (0.87) & 1.00e+0 (0.88) & 6.25e-6 (0.05) \\
  $2^{-10}$ & 3.22e-2 (1.09) & 1.79e-2 (1.08) & 4.79e-1 (1.07) & 4.20e-6 (0.57) \\
  $2^{-11}$ & 1.41e-2 (1.20) & 7.90e-3 (1.18) & 2.13e-1 (1.17) & 2.46e-6 (0.77) \\
  $2^{-12}$ & 6.22e-3 (1.18) & 3.48e-3 (1.18) & 9.37e-2 (1.19) & 1.32e-6 (0.89) \\
  $2^{-13}$ & 2.87e-3 (1.12) & 1.59e-3 (1.13) & 4.26e-2 (1.14) & 6.83e-7 (0.96) \\
  $2^{-14}$ & 1.37e-3 (1.07) & 7.57e-4 (1.07) & 2.02e-2 (1.08) & 3.45e-7 (0.98) \\
  $2^{-15}$ & 6.69e-4 (1.03) & 3.68e-4 (1.04) & 9.80e-3 (1.04) & 1.73e-7 (0.99) \\
  $2^{-16}$ & 3.31e-4 (1.02) & 1.82e-4 (1.02) & 4.83e-3 (1.02) & 8.69e-8 (1.00) \\
  $2^{-17}$ & 1.64e-4 (1.01) & 9.02e-5 (1.01) & 2.39e-3 (1.01) & 4.35e-8 (1.00) \\
  \bottomrule
  \end{tabular}
  \caption{Errors and convergence rates for the displacement field $u$ and scalar $s$.}
  \label{tab:error_u_s}
\end{table}
  Table \ref{tab:error_Q} presents the errors and  convergence rates between the numerical solutions and the benchmark solution for both the $\mathcal{Z}$-norm and  the $2$-norm.  The results indicate that the proposed SAV-EI scheme achieves approximately first-order convergence in time for the $Q$-tensor in all three norms. Table \ref{tab:error_u_s} shows the errors and convergence rates for the velocity field $\bu$ and the auxiliary variable $s$. The observed convergence rates are consistent with the theoretical predictions, confirming the first-order temporal accuracy of the scheme for both variables.

 \vskip 0.2cm
{\bf Acknowledgements.} G. Ji is partially supported by the National Natural Science Foundation of China (Grant No. 12471363).

\bibliographystyle{siam}
\bibliography{S0362546X14002934}

\begin{thebibliography}{10}

\bibitem{deGennes1974}
{\sc P.-G. de~Gennes}, {\em The Physics of Liquid Crystals}, Oxford University
  Press, Oxford, 1974.

\bibitem{du2018stabilized}
{\sc Q.~Du, L.~Ju, X.~Li, and Z.~Qiao}, {\em Stabilized linear semi-implicit
  schemes for the nonlocal cahn--hilliard equation}, Journal of Computational
  Physics, 363 (2018), pp.~39--54.

\bibitem{du2019}
\leavevmode\vrule height 2pt depth -1.6pt width 23pt, {\em Maximum principle
  preserving exponential time differencing schemes for the nonlocal allen--cahn
  equation}, SIAM Journal on Numerical Analysis, 57 (2019), pp.~875--898.

\bibitem{du2021}
{\sc Q.~{D}u, L.~Ju, X.~Li, and Z.~Qiao}, {\em Maximum bound principles for a
  class of semilinear parabolic equations and exponential time-differencing
  schemes}, SIAM Review, 63 (2021), pp.~317--359.

\bibitem{han2015microscopic}
{\sc J.~Han, Y.~Luo, W.~Wang, P.~Zhang, and Z.~Zhang}, {\em From microscopic
  theory to macroscopic theory: a systematic study on modeling for liquid
  crystals}, Archive for rational mechanics and analysis, 215 (2015),
  pp.~741--809.

\bibitem{hicks2024modelling}
{\sc A.~L. Hicks and S.~W. Walker}, {\em Modelling and simulation of the
  cholesteric landau-de gennes model}, Proceedings of the Royal Society A, 480
  (2024), p.~20230813.

\bibitem{hou2025energy}
{\sc D.~Hou, X.~Li, Z.~Qiao, and N.~Zheng}, {\em Energy stable and maximum
  bound principle preserving schemes for the-tensor flow of liquid crystals},
  SIAM Journal on Numerical Analysis, 63 (2025), pp.~854--880.

\bibitem{jiang2022improving}
{\sc M.~Jiang, Z.~Zhang, and J.~Zhao}, {\em Improving the accuracy and
  consistency of the scalar auxiliary variable (sav) method with relaxation},
  Journal of Computational Physics, 456 (2022), p.~110954.

\bibitem{ju2022generalized}
{\sc L.~Ju, X.~Li, and Z.~Qiao}, {\em Generalized sav-exponential integrator
  schemes for allen--cahn type gradient flows}, SIAM journal on numerical
  analysis, 60 (2022), pp.~1905--1931.

\bibitem{liu2025maximum}
{\sc Y.~Liu, C.~Quan, and D.~Wang}, {\em On the maximum bound principle and
  energy dissipation of exponential time differencing methods for the
  matrix-valued allen--cahn equation}, IMA Journal of Numerical Analysis, 45
  (2025), pp.~3342--3377.

\bibitem{liu2024novel}
{\sc Z.~Liu, Y.~Zhang, and X.~Li}, {\em A novel energy-optimized technique of
  sav-based (eop-sav) approaches for dissipative systems}, Journal of
  Scientific Computing, 101 (2024), p.~38.

\bibitem{pazy2012semigroups}
{\sc A.~Pazy}, {\em Semigroups of linear operators and applications to partial
  differential equations}, Springer Science \& Business Media, 2012.

\bibitem{shen2018scalar}
{\sc J.~Shen, J.~Xu, and J.~Yang}, {\em The scalar auxiliary variable (sav)
  approach for gradient flows}, Journal of Computational Physics, 353 (2018),
  pp.~407--416.

\bibitem{shen2019new}
\leavevmode\vrule height 2pt depth -1.6pt width 23pt, {\em A new class of
  efficient and robust energy stable schemes for gradient flows}, SIAM Review,
  61 (2019), pp.~474--506.

\bibitem{shi2025modified}
{\sc B.~Shi, Y.~Han, C.~Ma, A.~Majumdar, and L.~Zhang}, {\em A modified
  landau--de gennes theory for smectic liquid crystals: phase transitions and
  structural transitions}, SIAM Journal on Applied Mathematics, 85 (2025),
  pp.~821--847.

\bibitem{wang2021modelling}
{\sc W.~Wang, L.~Zhang, and P.~Zhang}, {\em Modelling and computation of liquid
  crystals}, Acta Numerica, 30 (2021), pp.~765--851.

\bibitem{xia2023variational}
{\sc J.~Xia and P.~E. Farrell}, {\em Variational and numerical analysis of a
  q-tensor model for smectic-a liquid crystals}, ESAIM: Mathematical Modelling
  and Numerical Analysis, 57 (2023), pp.~693--716.

\bibitem{xia2021structural}
{\sc J.~Xia, S.~MacLachlan, T.~J. Atherton, and P.~E. Farrell}, {\em Structural
  landscapes in geometrically frustrated smectics}, Physical review letters,
  126 (2021), p.~177801.

\bibitem{zhang2025novel}
{\sc B.~Zhang, C.~Zhou, and H.~Fu}, {\em A novel efficient generalized
  energy-optimized exponential sav scheme with variable-step bdfk method for
  gradient flows}, Applied Numerical Mathematics, 210 (2025), pp.~39--63.

\bibitem{zhang2022generalized}
{\sc Y.~Zhang and J.~Shen}, {\em A generalized sav approach with relaxation for
  dissipative systems}, Journal of Computational Physics, 464 (2022),
  p.~111311.

\end{thebibliography}

\end{document}